\documentclass[12pt]{article} % for a short document

\usepackage[margin=3cm]{geometry}
\usepackage[english]{babel}
\usepackage{hyperref}
\usepackage{amsmath}
\usepackage{amsthm}
\usepackage{amssymb}
\usepackage{ucs}
\usepackage{enumerate}
\usepackage[utf8x]{inputenc}
\usepackage[T1]{fontenc}
 
\usepackage{tikz}
\usepackage{tikz-cd}
\usetikzlibrary{arrows}
\usetikzlibrary{decorations.markings}
\pgfdeclarelayer{nodelayer}
\pgfdeclarelayer{edgelayer}
\pgfsetlayers{nodelayer,main,edgelayer}
\tikzstyle{none}=[draw=none]   
\tikzstyle{bigtiparrow}=[->,thick, >=angle 90]
\tikzstyle{bigtiparrow2}=[->,thick, >=angle 90,preaction={draw=white, -,line width=6pt}]
\tikzstyle{lrarrow}=[<->,thick, >=angle 90,preaction={draw=white, -,line width=6pt}]

\tikzstyle{new}=[rectangle,fill=white,draw=white, inner sep=2pt]
\tikzstyle{new2}=[rectangle,fill=white,draw=white, inner sep=6pt]

\DeclareMathOperator{\CCost}{\mathrm{C-Cost}}
\DeclareMathOperator{\rng}{\mathrm{rng}}
\DeclareMathOperator{\supp}{\mathrm{supp}}
\DeclareMathOperator{\Cost}{\mathrm{Cost}}
\DeclareMathOperator{\RR}{\mathcal{R}}
\DeclareMathOperator*{\esssup}{ess\,sup}

\newcommand{\Stild}{\widetilde{\mathfrak S}}
\newcommand{\MAlg}{\mathrm{MAlg}}
\newcommand{\Aut}{\mathrm{Aut}}

  \newcommand{\R}{\mathbb R}
  \newcommand{\N}{\mathbb N}
  
  \newcommand{\Z}{\mathbb Z}
  
  \newcommand{\LL}{\mathrm L}

 \newcommand{\dom}{\mathrm{dom}\;}

  \newcommand{\inv}{^{-1}}

  \renewcommand{\leq}{\leqslant}
  \renewcommand{\geq}{\geqslant}
  \newcommand{\abs}[1]{\left\lvert #1\right\rvert}
  \newcommand{\norm}[1]{\left\lVert #1\right\rVert}

  \newcommand{\la}{\left\langle}
  \newcommand{\ra}{\right\rangle}
  \newcommand{\into}{\hookrightarrow}

\newtheorem{thm}{Theorem}[section]
\newtheorem{cor}[thm]{Corollary}
\newtheorem{lem}[thm]{Lemma}
\newtheorem{prop}[thm]{Proposition}
\newtheorem*{claim}{Claim}

\theoremstyle{definition}
\newtheorem{step}{Step}
\newtheorem{qu}[thm]{Question}

\newtheorem{df}[thm]{Definition}
\newtheorem*{rmq}{Remark}

\newtheorem*{conj}{Conjecture}

\title{On full groups of non-ergodic probability measure-preserving equivalence relations}

\author{François Le Maître}
\date{}
\begin{document}

\maketitle

\begin{abstract}
This article generalizes the results of \cite{gentopergo} to the non-ergodic case by giving a formula relating the topological rank of the full group of an aperiodic pmp equivalence relation to the cost of its ergodic components. Furthermore, we obtain examples of full groups that have a dense free subgroup whose rank is equal to the topological rank of the full group, using a Baire category argument. We then study the automatic continuity property for full groups of aperiodic equivalence relations, and find a connected metric for which they have the automatic continuity property. This allows us to provide an algebraic characterization of aperiodicity for pmp equivalence relations, namely the non-existence of homomorphisms from their full groups into totally disconnected separable groups. A simple proof of the extreme amenability of full groups of hyperfinite pmp equivalence relations is also given, generalizing a result of Giordano and Pestov to the non-ergodic case  \cite[Thm. 5.7]{MR2311665}.
\end{abstract}

\tableofcontents

\section{Introduction}

When considering dynamical systems given by a measure-preserving bijection on a standard probability space, one has at hand various invariants which sometimes allow to classify them up to conjugacy. For instance, the entropy is a complete invariant for Bernoulli shifts, meaning that two Bernoulli shifts are conjugate iff they share the same entropy.  An easier to define invariant is the partition of the probability space induced by the orbits of the dynamical systems. This partition turns out to be as far as possible from a complete invariant: a result of Dye asserts that up to a measure-preserving bijection, all ergodic dynamical systems induce the same partition of the space almost everywhere \cite{MR0131516}. In other words, they are all \textbf{orbit equivalent}. 

Orbit equivalence is thus trivial in the realm of classical ergodic theory, that is, ergodic theory of probability measure-preserving $\Z$-action. We then move to the setting of probability measure-preserving (pmp) actions of  arbitrary countable groups $\Gamma$, where the picture gets much richer. Indeed, whenever $\Gamma$ is a non-amenable group, Epstein has shown that it admits a continuum of non orbit equivalent pmp ergodic free actions \cite{MR2711968}. On the other hand, generalizing the aforementioned result of Dye, Ornstein and Weiss have shown that any two pmp ergodic actions of any two amenable groups are orbit equivalent \cite{MR551753}. 

Given these results, it is desirable to have a more precise understanding of the situation for pmp ergodic free actions of non-amenable groups. But given a non-amenable group $\Gamma$, there are currently no invariants which yield an infinite family of non orbit equivalent pmp free ergodic $\Gamma$-actions. However, if one wants to distinguish the partitions induced by pmp free ergodic actions of two different non-amenable groups, there is an invariant which might do the trick, namely the cost.

Introduced by Levitt in 1995 \cite{MR1366313}, the cost of a pmp action is roughly the minimal number of elements needed to generate the partition of the space into orbits induced by this action. In 2000, Gaboriau showed that all the pmp free actions of the free group on $n$ generators have cost $n$ \cite{MR1728876}, so that free groups of different ranks can never have orbit equivalent pmp free actions. As we will see later, the cost is deeply linked to another natural invariant of pmp actions: their full group, which we now define.

Let us fix a standard probability space $(X,\mu)$ and denote by $\Aut(X,\mu)$ the group of its measure-preserving bijections, identified up to null sets. Given a non-trivial pmp action of a countable group $\Gamma$ on $(X,\mu)$, one associates to it a \textbf{pmp equivalence relation} $\mathcal R_\Gamma$ defined by $x\mathcal R_\Gamma y$ iff $x\in\Gamma y$. Because the equivalence classes of $\mathcal R_\Gamma$ are precisely the $\Gamma$-orbits, orbit equivalence can be rephrased as the study of such pmp equivalence relations up to isomorphism. Now if $\mathcal R$ is a pmp equivalence relation, its \textbf{full group}, denoted by $[\mathcal R]$ is defined by
$$[\mathcal R]=\{T\in\Aut(X,\mu): \forall x\in X, T(x)\mathcal R x\}.$$

The full group of a pmp equivalence relation is a Polish  group when equipped with the uniform metric $d_u(T,T')=\mu(\{x\in X: T(x)\neq T'(x)\})$, homeomorphic as a topological space to an infinite-dimensional Hilbert space \cite[Thm. 1.5]{MR2599891}. 

Suppose that $T:X\to X$ is an isomorphism between two pmp equivalence relations $\mathcal R$ and $\mathcal R'$, i.e. a class measure-preserving bijection of $(X,\mu)$ such that for almost every $x\in X$, $T([x]_\mathcal R)=[T(x)]_{\mathcal R'}$. Then $T$ conjugates the full group of $\mathcal R$ to the full group of $\mathcal R'$, which are consequently isomorphic as topological groups. In other words, the full group is an invariant of orbit equivalence. 

We thus leave the realm of orbit equivalence and are led to study these full groups as topological groups in their own right, which was the main motivation for the already mentioned article of Kittrell and Tsankov \cite{MR2599891}. Furthermore, Dye's reconstruction theorem implies that full groups are a complete invariant of orbit equivalence: any abstract group isomorphism between full groups must descend from a class measure-preserving orbit equivalence between the corresponding equivalence relations, as long as one of the two pmp equivalence relations has no finite orbits.

Let us now give examples where one can see concretely how the full group ‘‘remembers'' the pmp equivalence relation. 

\begin{thm}[{Eigen \cite{MR654590}, Dye \cite[Prop. 5.1]{MR0158048}}] Let $\mathcal R$ be a pmp equivalence relation. Then $\mathcal R$ is ergodic iff $[\mathcal R]$ is simple. Moreover, the closed normal subgroups of $[\mathcal R]$ are of the form 
$$[\mathcal R_A]=\{T\in[\mathcal R]: \supp T\subseteq A\},$$
where $A$ is an $\mathcal R$-invariant subset of $X$. 
\end{thm}

\begin{thm}[{Giordano-Pestov \cite[Thm. 5.7]{MR2311665}}]
Let $\mathcal R$ be an ergodic equivalence relation. Then $\mathcal R$ is amenable iff $[\mathcal R]$ is extremely amenable.
\end{thm}

\begin{thm}[{Le Maître \cite[Thm. 1]{gentopergo}}]
Let $\mathcal R$ be a pmp ergodic equivalence relation. Then the topological rank\footnote{By definition, the topological rank $t(G)$ of a topological group $G$ is the minimal number of elements needed to generate a dense subgroup of $G$.} $t([\mathcal R])$ of the full group of $\mathcal R$ is related to the integer part of the cost of $\mathcal R$ by the formula
$$t([\mathcal R])=\lfloor\Cost(\mathcal R)\rfloor+1.$$
\end{thm}

This article is an attempt to further understand algebraic and topological properties of full groups, especially in the non-ergodic case. The results are extracted from the author's thesis.

\subsection{Topological rank for full groups}

Our main result is a direct continuation of \cite{gentopergo}, where the question of the topological rank  of the full group of a pmp ergodic equivalence relation $\mathcal R$ was addressed, using previous works of Kittrell, Tsankov \cite{MR2599891} and Matui \cite{MR3103094}. We may first note that when a pmp equivalence relation $\mathcal R$ has all its orbits of cardinality $\leq n$ for a fixed $n\in\N$, its full group cannot be topologically generated by a finite number of elements, for it is a locally finite group (i.e. all finitely generated subgroups are finite). Moreover, if there is a positive $\mathcal R$-invariant set $A$ onto which $\mathcal R$ has  only finite equivalence classes, then in particular there exists $n\in\N$ and a smaller positive $\mathcal R$-invariant set $B$ onto which $\mathcal R$ has has all its orbits of cardinality $\leq n$. Then the full group of $\mathcal R$ factors continuously onto the full group of its restriction to $B$, which is not topologically finitely generated by the above observation, so that the full group of $\mathcal R$ is not topologically finitely generated. So we can restrict ourselves to \textbf{aperiodic} pmp equivalence relations, that is, having only infinite classes. 

In this article, we compute the topological rank of full groups of pmp aperiodic equivalence relations, using the same line of ideas as in \cite{gentopergo}, but with some modifications which we now briefly explain. First, in \cite{gentopergo} we used a result of Matui \cite[Thm. 3.2]{MR3103094} which gives two topological generators for the full group of the \textit{ergodic} hyperfinite equivalence relation\footnote{Note that Andrew Marks has now a much shorter proof of Matui's result, available as a note \href{http://www.its.caltech.edu/~marks/notes/top_gen.pdf}{on his website}.}, and so we had to generalize Matui's theorem to non-ergodic hyperfinite equivalence relations (Theorem \ref{thmgenEO}). 

Second, instead of working with the ergodic decomposition, we chose a more ‘‘functional'' approach which was already present in Dye's founding paper \cite{MR0131516}. If $\mathcal R$ is a pmp equivalence relation, we have a conditional expectation which turns every measurable map into an $M_{\mathcal R}$-measurable map, where $M_{\mathcal R}$ is the $\sigma$-algebra of all $\mathcal R$-invariant sets. This yields what we call the $\mathcal R$-\textbf{conditional measure} of a set $A$, defined to be the conditional expectation of its characteristic function $\chi_A$. Such a conditional measure can of course be understood as $x\mapsto \mu_x(A)$, where $(\mu_x)_{x\in X}$ is the ergodic decomposition of $\mathcal R$.  Also note that when $\mathcal R$ is ergodic, the $\mathcal R$-conditional measure equals $\mu$. Every element of $[\mathcal R]$ preserves the $\mathcal R$-conditional measure and conversely, if $A,B\subseteq X$ have the same $\mathcal R$-conditional measure, then there is $T\in[\mathcal R]$ such that $T(A)=B$ (Proposition \ref{transitive}). More details can be found in subsection \ref{condmesdefs}.

In the same spirit, we introduce a notion of cost for non-ergodic equivalence relations which can be undersood as a function mapping an ergodic component to its cost, which we refer to as conditional cost. Again rather than using the ergodic decomposition, we see the  conditional cost of a pmp equivalence relation $\mathcal R$ as an $M_{\mathcal R}$-measurable function from $X$ to $\R^+\cup\{+\infty\}$. Also note that when $\mathcal R$ is ergodic, the conditional cost is the usual cost. For more details, cf. subsection \ref{CCost}.

We can now state our main theorem:
\begin{thm}\label{thethm}
Let $\mathcal R$ be an aperiodic pmp equivalence relation. Then the topological rank $t([\mathcal R])$ of the full group of $\mathcal R$ is related to the conditional cost of $\mathcal R$ by the relation

$$t([\mathcal R])=\esssup_{x\in X} \lfloor \CCost(\mathcal R)(x)\rfloor+1.$$
\end{thm}

Because every countable product of full groups can itself be seen as a full group, we have the following corollary. 

\begin{cor}
Let $(\mathcal R_i)_{i\in I}$ be a countable family of pmp equivalence relations. Then 
$$t\left(\prod_{i\in I}[\mathcal R_i]\right)=\sup_{i\in I}t([\mathcal R_i]).$$
\end{cor}

\begin{rmq}Theorem \ref{thethm} also allows us to remove the ergodicity assumption in corollaries 2 and 3 of \cite{gentopergo}.
\end{rmq}

\subsection{(Dense) free groups in full groups}

Recall that $T\in\Aut(X,\mu)$ is called \textbf{aperiodic} if all its orbits are infinite, and that the set of such elements, denoted by $\mathrm{APER}$, is closed for the uniform metric. Our next result goes into the direction of a better understanding of free groups in full groups, and generalizes \cite[Thm. 3.7]{MR2583950} (cf. Theorem \ref{freegp}).

\begin{thm}Let $\mathcal R$ be a pmp equivalence relation, and let $T\in[\mathcal R]$ have infinite order. Then for all $n\in\N$ we have the following:\begin{enumerate}[(1)]
\item The set of $n$-uples $(U_1,...,U_n)$ of elements of $[\mathcal R]$ such that
$$\la T,U_1,...,U_n\ra \simeq \mathbb F_{n+1}$$
is a dense $G_\delta$ in $[\mathcal R]^n$.
\item If furthermore $\mathcal R$ is aperiodic, then the set of $n$-uples of aperiodic elements $(U_1,...,U_n)$  of $[\mathcal R]$ such that 
$$\la T,U_1,...,U_n\ra \simeq \mathbb F_{n+1}$$
is a dense $G_\delta$ in $([\mathcal R]\cap \mathrm{APER})^{n}$. 
\end{enumerate}
\end{thm}

Combining this theorem with a density result for topological generators in aperiodic pmp equivalence relations of cost one (Proposition \ref{littlegdelta}), we get a good understanding of topological generators for full groups of aperiodic pmp equivalence relations of cost one (Theorem \ref{Gdelta}).

\begin{thm}\label{gdelta}
Let $\mathcal R$ be any pmp equivalence relation, and consider the set $$G=\{(T,U)\in (\mathrm{APER}\cap [\mathcal R])\times[\mathcal R]: \overline{\la T,U\ra}=[\mathcal R]\text{ and } \la T,U\ra\simeq\mathbb F_2\}.$$
Then the following assertions are equivalent.
\begin{enumerate}[(1)]
\item $\mathcal R$ is aperiodic of cost one.
\item $G$ is a dense $G_\delta$ in $(\mathrm{APER}\cap [\mathcal R])\times[\mathcal R]$.
\end{enumerate}
\end{thm}

Note that the the aperiodicity restriction on the first coordinate is necessary. Indeed, topological generators of $[\mathcal R]$ have to generate the equivalence relation $\mathcal R$, and if $D$ is any dense subset of $[\mathcal R]\times[\mathcal R]$, it contains elements of arbitrarily small support. Then, by definition of the cost some elements of $D$ do not generate the equivalence relation $\mathcal R$, hence cannot be topological generators of $[\mathcal R]$. 

When $\mathcal R$ is generated by a single pmp automorphism\footnote{By a theorem of Dye \cite[Thm. 1]{MR0131516}, a pmp equivalence relation $\mathcal R$ is singly-generated iff it is hyperfinite (cf. definition \ref{dfhyperfinite}), which in turn is equivalent to the amenability of $\mathcal R$ \cite{MR662736}.}, the set $\mathrm{GEN}(\mathcal R)$ of generators of $\mathcal R$ is a dense $G_\delta$ in $\mathrm{APER}\cap[\mathcal R]$, so that in this case we have the following reformulation of Theorem \ref{gdelta}.
\begin{cor}
Let $\mathcal R$ be any aperiodic pmp singly-generated equivalence relation. Then the set $$\{(T,U)\in \mathrm{GEN}(\mathcal R)\times[\mathcal R]: \overline{\la T,U\ra}=[\mathcal R]\text{ and } \la T,U\ra=\mathbb F_2\}$$
is a dense $G_\delta$ in $\mathrm{GEN}(\mathcal R)\times[\mathcal R]$.
\end{cor}

Using the Kuratowski-Ulam theorem, we see that for $\mathcal R$ singly-generated aperiodic, the set of $T\in \mathrm{GEN}(\mathcal R)$ such that 
\[\{U\in[\mathcal R]: \overline{\la T,U\ra}=[\mathcal R]\text{ and } \la T,U\ra=\mathbb F_2\}\text{ is a dense }G_\delta\text{ in }[\mathcal R]\]
is itself a dense $G_\delta$ in $\mathrm{GEN}(\mathcal R)$. But it might actually be equal to $\mathrm{GEN}(\mathcal R)$, which yields the following general question.

\begin{qu} Let $T$ be an aperiodic element of $\Aut(X,\mu)$, let $\mathcal R_T$ be the pmp equivalence relation it generates. Does there exist $U\in[\mathcal R_T]$ such that $\overline{\la T,U\ra}=[\mathcal R_T]$? If yes, is the set of such $U$'s a dense $G_\delta$?
\end{qu}

\begin{rmq}In the author's thesis, it is shown that the answer to both these questions is yes for every $T\in\Aut(X,\mu)$ of rank one, constructing a $U\in[\mathcal R_T]$ very similar to the one we have here for the odometer (Theorem \ref{thmgenEO}). But this idea seems to fail for infinite rank transformations like Bernoulli shifts.
\end{rmq}
When $\mathcal R$ is aperiodic of cost one, Theorem \ref{gdelta} also gives a positive answer to the following question, which was suggested to the author by Gelander. 

\begin{qu}Let $\mathcal R$ be an aperiodic pmp equivalence relation. Does it have a dense free subgroup of rank $t([\mathcal R])$?
\end{qu}
We do not even have an \textit{explicit}  dense free subgroup of rank 2 in the full group of the hyperfinite ergodic equivalence relation. Let us end this section by mentioning a very interesting conjecture of Thom about free groups in full groups.

\begin{conj}[Thom]There are no discrete free subgroups of rank 2 in the full group of the hyperfinite ergodic equivalence relation.
\end{conj}

\subsection{Automatic continuity for a finer metric}

Dye's reconstruction theorem \cite[Thm. 2]{MR0158048} states that whenever a pmp equivalence relation $\mathcal R$ is aperiodic, any abstract group isomorphism $\psi$ between the full groups of $\mathcal R$ and another pmp equivalence relation $\mathcal R'$ is implemented by a class measure-preserving orbit equivalence $T\in\Aut^*(X,\mu)$ between $\mathcal R$ and $\mathcal R'$, i.e. we have $\psi(g)=TgT\inv$ for all $g\in [\RR]$ and $T\times T(\mathcal R)=\mathcal R'$. Morevoer, if $\mathcal R$ is ergodic,  $T$ is automatically measure-preserving. So in the ergodic case Dye's theorem implies that any abstract group isomorphism between full groups is an isometry, while in the non-ergodic case it states that it is a homeomorphism. This was the motivation for Kittrell and Tsankov to study the automatic continuity property for full groups. 

\begin{df}Let $G$ be a topological group. It satisfies the \textbf{automatic continuity property} if every homomorphism from $G$ to a separable topological group is continuous.
\end{df}

\begin{thm}[{\cite[Thm. 1.2]{MR2599891}}]\label{KTcontaut}Let $\mathcal R$ be a pmp ergodic equivalence relation. Then the full group of $\mathcal R$ endowed with the uniform metric satisfies the automatic continuity property.
\end{thm}

It is a classical fact that whenever $G$ is a Polish group having the automatic continuity property, every isomorphism between $G$ and a Polish group $H$ is a homeomorphism (cf. for instance \cite{MR2535429}). The preceding theorem can thus be seen as a partial generalization of Dye's reconstruction theorem, and also has the following important consequence. 

\begin{cor}[{\cite[Thm. 1.2]{MR2599891}}] Every full groups of a pmp ergodic equivalence relations carries a unique Polish group topology, namely the uniform topology. 
\end{cor}

Note that whenever $\mathcal R$ is not aperiodic and non-trivial, its full group does not have the automatic continuity property, for it is connected but admits a non-trivial morphism into $\Z/2\Z$ (cf. Corollary \ref{caraper} and the paragraph following it). So a natural question is: what happens in the aperiodic case\footnote{When $\mathcal R$ has finitely many ergodic components, its full group is the finite product of the full groups of its restrictions to these components. But by Theorem \ref{KTcontaut}, each of these has the automatic continuity property, hence the full group of $\mathcal R$ itself has the automatic continuity. So we may restrict ourselves to aperiodic pmp equivalence relations having either a continuous or a countably infinite ergodic decomposition.}?

We only provide a very partial answer to this question (Theorem \ref{autocont}), namely we define a finer connected metric $d_C$ on the full group of $\mathcal R$ such that $([\mathcal R],d_{C})$ has the automatic continuity property whenever $[\mathcal R]$ is aperiodic. This metric can be thought of as uniform convergence along ergodic components, so that in the ergodic case it is the uniform metric, but it general it is not separable. Our result is thus a generalization of Theorem \ref{KTcontaut}, but the proof is very similar, and we think the metric $d_C$ is actually the right setting for Kittrell and Tsankov's result.

We use this result to give a characterization of aperiodic equivalence relations in terms of their full groups (Corollary \ref{caraper}).

\begin{thm}Let $\mathcal R$ be a pmp equivalence relation. Then the following assertions are equivalent.
\begin{enumerate}[(1)]
\item $\mathcal R$ is aperiodic. 
\item $[\mathcal R]$ has no nontrivial morphisms into $\Z/2\Z$.
\item $[\mathcal R]$ has no nontrivial morphisms into totally disconnected separable groups.
\end{enumerate}
\end{thm}

\subsection{Extreme amenability}

Our last result (Theorem \ref{caraextam}) generalizes a result of T.\ Giordano and V.\ Pestov on the relationship between the amenability of a pmp equivalence relation and the extreme amenability of its full group. 
\begin{df}
A topological group $G$ is \textbf{extremely amenable} if every continuous $G$-action on a compact Hausdorff space $K$ admits a fixed point.
\end{df}

Extreme amenability is a very strong property, specific to ‘‘infinite dimensional groups'', for no non-trivial locally compact group can be extremely amenable \cite{MR0467705}. The terminology is justified by the following definition of amenability for topological groups:  a topological group $G$ is \textbf{amenable} if every continuous $G$-action  \textit{by affine isometries} on a compact \textit{convex subset of a locally convex topological vector space} admits a fixed point. 

Amenability for discrete groups admits a lot of different reformulations, some of which yield different characterizations of amenability for equivalence relations, and we refer the reader to \cite{MR1485618} for these. For our purposes, the following definition is the most relevant, and turns out to be equivalent to amenability by a celebrated result of Connes, Feldman and Weiss \cite{MR662736}. Recall that a pmp equivalence relation is \textbf{finite} if all its classes are finite.

\begin{df}\label{dfhyperfinite}A pmp equivalence relation $\mathcal R$ is \textbf{hyperfinite} if it can be written as a countable increasing union of finite pmp equivalence relations.
\end{df}

Giordano and Pestov showed in \cite[Prop. 5.3]{MR2311665} that an ergodic pmp equivalence relation is hyperfinite iff its full group is extremely amenable. We generalize this to the non-ergodic setting and give an easier proof (cf. Theorem \ref{caraextam}). One should note however that our proof does not yield that the full group of an amenable equivalence relation is a Levy group, while theirs does. Also, we still rely on the concentration of measure phenomenon via a theorem of Glasner \cite[Thm. 1.3]{MR1617456}, and it would be nice to have a proof which does not.

\subsection{Notations and preliminaries}

Everything will take place ‘‘modulo sets of measure zero''. A \textbf{cycle} is a measure-preserving automorphism of $(X,\mu)$ that has only finite orbits. A cycle is \textbf{odd} when all its orbits have an odd cardinality. We keep the same notations and definitions as in \cite[sec. 1 and 2]{gentopergo}, which we now briefly recall.

If $(X,\mu)$ is a standard probability space, and $A,B$ are Borel subsets of $X$, a \textbf{partial isomorphism} of $(X,\mu)$ of \textbf{domain} $A$ and \textbf{range} $B$ is a Borel bijection $f: A\to B$ which is measure-preserving for the measures induced by $\mu$ on $A$ and $B$ respectively. We denote by $\dom f=A$ its domain, and by $\rng f=B$ its range. Note that in particular, $\mu(\dom f)=\mu(\rng f)$. A \textbf{graphing} is a countable set of partial isomorphisms of $(X,\mu)$, denoted by $\Phi=\{\varphi_1,...,\varphi_k,...\}$ where the $\varphi_k$'s are partial isomorphisms. It \textbf{generates} a pmp equivalence relation $\mathcal R_\Phi$, defined to be the smallest equivalence relation containing $(x,\varphi(x))$ for every $\varphi\in \Phi$ and $x\in\dom\varphi$.

The \textbf{pseudo full group} of $\mathcal R$, denoted by $[[\mathcal R]]$, consists of all partial isomorphisms $\varphi$ such that $\varphi(x)\, \mathcal R \, x$ for all $x\in \dom\varphi$.

Let $p\in\N$. A \textbf{pre}-$p$-\textbf{cycle} is a graphing $\Phi=\{\varphi_1,...,\varphi_{p-1}\}$ such that  the following two conditions  are satisfied:
\begin{enumerate}[(i)]\item  $\forall i\in\{1,...,p-1\}, \rng\varphi_i=\dom\varphi_{i+1}$.
\item The following sets are all disjoint:
 $$\dom\varphi_1, \dom \varphi_2,...,\dom\varphi_{p-1},\rng\varphi_{p-1}.$$
\end{enumerate}
A $p$\textbf{-cycle} is an element $C\in \Aut(X,\mu)$ whose orbits have cardinality $1$ or $p$. So according to the terminology introduced at the beginning of this section, a $p$-cycle is a cycle which is odd iff $p$ is odd.

Given a pre-$p$-cycle $\Phi=\{\varphi_1,...,\varphi_{p-1}\}$, we can extend it to a $p$-cycle $C_\Phi\in\Aut(X,\mu)$ as follows:
$$C_\Phi(x)=\left\{\begin{array}{ll}\varphi_i(x) & \text{if }x\in \dom\varphi_i\text{ for some }i<p, \\
\varphi_1\inv\varphi_2\inv\cdots\varphi_{p-1}\inv(x) &\text{if }x\in \rng\varphi_{p-1},\\
x & \text{otherwise.}\end{array}\right.$$

The following theorem will again prove very useful. 
\begin{thm}[\cite{MR2599891}, Thm. 4.7]\label{ktdense}
Let $\mathcal R_1$, $\mathcal R_2$,... be measure-preserving equivalence relations on $(X,\mu)$, and let $\mathcal R$ be their join (i.e. the smallest equivalence relation containing all of them). Then $\la\bigcup_{n\in\N}[\mathcal R_n]\ra$ is dense in $[\mathcal R]$.
\end{thm}

%If $f$ and $g$ are two elements of $\LL^\infty(X,\mu)$, say $f\leq g$ if for all $x\in X$, $f(x)\leq g(x)$, and $f<g$ if $f\leq g$ and there exists a positive set $A$ such that $f(x)<g(x)$ for all $x\in A$, that is, $f\leq g$ but $f\neq g$. When we say a function $f$ is \textbf{positive}, we mean that $f>0$, and we will say that $f$ is \textbf{everywhere positive} if for all $x\in X$, $f(x)>0$. 

\section{Conditional expectation and non-ergodic equivalence relations}\label{condexp}

\subsection{Conditional expectation and conditional measure}\label{condmesdefs}

Given a pmp equivalence relation $\mathcal R$, one may look at the measure subalgebra of the $\mathcal R$-invariant sets, i. e. the set of all $A\in \MAlg(X,\mu)$ such that for any $\varphi\in [\mathcal R]$, $\varphi\inv(A)=A$. We denote this closed subalgebra by $M_{\mathcal R}$, and by $\LL^2_{\mathcal R}(X)$ the set of square integrable $M_{\mathcal R}$-measurable real functions.

The orthogonal projection $\mathbb E_\mathcal R$ from the Hilbert space $\LL^2(X)$ onto the closed subspace $\LL^2_{\mathcal R}(X)$ satisfies the following equality, which defines it uniquely: for any $f\in \LL^2(X)$ and $g\in \LL^2_{\mathcal R}(X)$,
$$\int_X fg=\int_X\mathbb E_{\mathcal R}(f)g.$$
$\mathbb E_{\mathcal R}$ is called a \textbf{conditional expectation}. When $A$ is a subset of $X$, its characteristic function is an element of $\LL^2(X)$, and we call $\mathbb E_{\mathcal R}(\chi_A)$ the $\mathcal R$-\textbf{conditional measure} of $A$, denoted by $\mu_{\mathcal R}(A)$. Because $\mathbb E_{\mathcal R}$ is a contraction for the $\LL^\infty$ norm, $\mu_{\mathcal R}(A)$ is an $M_\mathcal R$-measurable function taking values in $[0,1]$. Proposition \ref{maha} roughly states that any $M_{\mathcal R}$-measurable function $f:(X,\mu)\to[0,1]$ is equal to $\mu_{\mathcal R}(A)$ for some $A\subseteq X$. 

We will use the partial order on $\LL^2(X)$ defined by $f\leq g$ if $f(x)\leq g(x)$ for almost every $x\in X$. Note that by $f<g$, we mean that $f\leq g$ and $f\neq g$, and not that $f(x)< g(x)$ for almost every $x\in X$. \\

%Let $\mathcal R$ be a pmp equivalence relation on $(X,\mu)$. We say it is of type I if there exists a Borel set $A\subseteq X$ such that for  all $x\in X$, there is a unique $a\in A$ such that $x \,\mathcal R\,a$. Such an $A$ is called a \textbf{transversal} of $\mathcal R$. It is a standard fact that type I equivalence relations are \textbf{finite}, i.e. only have finite equivalence classes. One can also look at restrictions of $\mathcal R$ to non null sets $A$, and if none of these is finite we say $\mathcal R$ is of type $\mathrm{II}_1$. This condition is equivalent to being \textbf{aperiodic}, i.e. to having only infinite classes. 
The following proposition is due to Dye \cite{MR0131516}; we give a simple proof, based on the marker lemma which we now recall.

\begin{lem}[{\cite[Lem. 6.7]{MR2095154}}]\label{marker} Let $\mathcal R$ be an aperiodic equivalence relation, then there exists a decreasing sequence of $A_n\subseteq X$ which intersect every $\mathcal R$ class and such that $\bigcap_n A_n=\emptyset$. 
\end{lem}

We call such $A_n$'s a sequence of \textbf{markers}.

\begin{prop}[\cite{MR0131516}, Maharam's lemma]\label{maha}
$\mathcal R$ is aperiodic iff for any $A\subseteq X$, and for any $M_\mathcal R$-measurable function $f$ such that $0\leq f\leq \mathbb \mu_{\mathcal R}(A)$, there exists $B\subseteq A$ such that the $\mathcal R$-conditional measure of $B$ equals $f$.
\end{prop}
\begin{proof}
First note that we can restrict ourselves to the case $0<f\leq \mu_{\mathcal R}(A)$. By a maximality argument\footnote{Such a maximality argument goes as follows: one builds by transfinite induction a non-decreasing sequence $(B_\alpha)_{\alpha<\omega_1}$ of Borel subsets of $A$  such that for all $\alpha<\omega_1$, $\mu_{\mathcal R}(B_\alpha)\leq f$ and whenever $\mu_\mathcal R(B_\alpha)< f$, then $\mu_{\mathcal R}(B_{\alpha+1})>\mu_{\mathcal R}(B_\alpha)$. Such a sequence will then have to be stationnary, for there is no uncountable disjoint family of Borel sets of positive measure. But that means that there is some $\alpha<\omega_1$ for which  $\mu_\mathcal R(B_\alpha)=f$. See \cite[Prop. D.1]{lemai2014} for another approach based on the measure algebra of $(X,\mu)$. }, it suffices to show that we can find $B\subseteq A$ such that $0<\mu_{\mathcal R}(B)\leq f$. Let $A_n$ be a sequence of markers for $\mathcal R_{\restriction A}$, note that $ \mathbb \mu(A_n)=\int \mu_{\mathcal R}(A_n)\to 0$, so up to taking a subsequence we can assume $\mu_{\mathcal R}(A_n)$ converges pointwise to 0. Now define $B$ to be the set of $x\in X$ such that here exists an integer $n$  for which $x\in A_n$ and $\mu_{\mathcal R}(A_n)(x)<f(x)$. 

Conversely, if $\mathcal R_{\restriction A}$ is finite, one can further restrict $A$ and suppose all its classes are of size $n$. Now the function $\frac 1{2n}$ cannot arise as the conditional measure of $B\subseteq A$. 
\end{proof}

\begin{rmq}\label{split}The above proposition has the following nice consequence: given $A\subseteq X$, we can split it ‘‘equally among ergodic components'': one puts $f= \mu_{\mathcal R}(A)/2$ and gets $B\subseteq A$ such that $\mu_{\mathcal R}(B)=\mu_{\mathcal R}(A\setminus B)=\mu_{\mathcal R}(A)/2$ (in particular, $\mu(B)=\frac{\mu(A)}2$). And of course one can also split $A$ in any finite number of pieces, all which have the same conditional measure.
\end{rmq}

\begin{prop}\label{transitive}Let $\mathcal R$ be a pmp equivalence relation. Then if two sets $A$ and $B$ have the same $\mathcal R$-conditional measure, there exists $\varphi\in[[\mathcal R]]$ whose domain is $A$ and whose range is $B$.
\end{prop}
\begin{proof}This is a standard maximality argument. One can for instance use the construction in  \cite[Lem. 7.10]{MR2095154}, and check that if $\mu(A\setminus A')>0$, then the $\mathcal R$-conditional measures of $A\setminus A'$ and $B\setminus B'$ are the same and positive. But then the $\mathcal R$-saturation of $A\setminus A'$ contains $B\setminus B'$, so that there exists a minimal $n$ such that $\gamma_n\inv(B\setminus B')\cap (A\setminus A')$ is non-null, a contradiction.
\end{proof}

\subsection{A non-ergodic analogue of Zimmer's lemma}
%Ref pour measure algebras (pt ê plus tôt)
Recall that every closed subalgebra of $\MAlg(X,\mu)$ arises as the inverse image of $\MAlg(Y,\nu)$ by some measure-preserving $\pi: (X,\mu)\to(Y,\nu)$, where $(Y,\nu)$ is a standard probability space possibly with atoms. So if $\mathcal R$ is a measure-preserving equivalence relation, one can find a standard probability space $(Y_{\mathcal R},\nu_{\mathcal R})$, possibly with atoms, and a measure-preserving map $\pi_{\mathcal R}: (X,\mu)\to (Y_\mathcal R,\nu_{\mathcal R})$ such that $\pi_{\mathcal R}\inv(\MAlg(Y_{\mathcal R}, \nu_{\mathcal R}))=M_{\mathcal R}$. This yields an identification between $\LL^2_{\mathcal R}(X)$ and $\LL^2(Y_{\mathcal R},\nu_{\mathcal R})$. 

The following proposition is essentially due to Dye \cite[Thm.\ 4 and 5]{MR0131516}, and generalizes the well-known fact often attributed to Zimmer that any ergodic equivalence relation contains an ergodic subequivalence relation generated by a single automorphism. Before we state this proposition, recall that the pmp equivalence relation $\mathcal R_0$ is the equivalence relation induced by the odometer $T_0$ acting on $\{0,1\}^\N$ equipped with the Bernoulli 1/2 product measure (see Section \ref{secro} for more details).

\begin{prop}[Dye]\label{zimnonergo}Let $\mathcal R$ be a pmp aperiodic equivalence relation on $(X,\mu)$, let $\pi_{\mathcal R}: (X,\mu)\to (Y_{\mathcal R},\nu_{\mathcal R})$ be the measure-preserving map associated with the algebra of $\mathcal R$-invariant subsets of $X$. Then there exists an isomorphism $\varphi: (X,\mu)\to (Y_{\mathcal R}\times2^\N,\nu_{\mathcal R}\times\lambda)$ such that $\mathrm{id}_{Y_{\mathcal R}}\times \mathcal R_0$ is mapped into $\mathcal R$ by $\varphi\inv$, and the following diagram commutes:
$$\begin{tikzcd}[column sep=0ex]
(X,\mu)
  \arrow{rr}{\varphi}\arrow[rightarrow]{rd}[description]{\pi_{\mathcal R}}& &(Y_{\mathcal R}\times 2^\N, \nu_{\mathcal R}\times\lambda)\arrow{ld}[description]{\pi}\\
      &(Y_{\mathcal R},\nu_{\mathcal R}) &
 \end{tikzcd}$$
\end{prop}

\subsection{Conditional cost}\label{CCost}

Let $\mathcal R$ be an aperiodic equivalence relation. As orbit equivalence remembers the ergodic decomposition of $\mathcal R$, the following notion makes sense as an invariant for non-ergodic equivalence relations. It uses the conditional measure $\mu_\mathcal R$, defined in Section \ref{condmesdefs}.

\begin{df}
The \textbf{C-cost} (conditionnal cost) of a pmp equivalence relation $\mathcal R$ is the infimum over all graphings $\Phi=\{\varphi_i\}_{i\in\N}$ which generate $\mathcal R$ of the functions
$$\CCost(\Phi)=\sum_{i\in\N}\mu_{\mathcal R}(\varphi_i).$$
\end{df}

Let us check that this definition makes sense. The functions $\mu_{\mathcal R}(\Phi)$ are positive elements of the von Neumann algebra $\LL^\infty_{\mathcal R}(X,\mu)=\LL^{\infty}(Y_{\mathcal R},\nu_{\mathcal R})$, to show that they have an infimum we must check that they constitute a directed set. But if $\Phi$ and $\Phi'$ are two graphings which generate $\mathcal R$, the set $A=\{x\in X: \CCost(\Phi)(x)<\CCost(\Phi')(x)\}$ is $\mathcal R$-invariant, and one can define a new graphing $\Psi$ made of   the restrictions of the elements of $\Phi$ to $A$ and the restrictions of the elements of $\Phi'$ to $X\setminus A$. Such a $\Psi$ still generates $\mathcal R$, and we have
$$\CCost(\Psi)=\inf(\CCost(\Phi),\CCost(\Phi')).$$

Using a maximality argument, one can then see that for every $M_\mathcal R$-measurable everywhere positive function $f$, there exists a graphing $\Phi$ which generates $\mathcal R$ such that its C-cost is less than $f+\CCost\mathcal R$.

We now give a non-ergodic version of Lemma III.5 in \cite{MR1728876}.
\begin{lem}\label{troispointcinq}Let $\mathcal R$ be an aperiodic pmp equivalence relation, let $\mathcal R_0^Y\subseteq \mathcal R$ be hyperfinite with the same algebra of invariant sets as $\mathcal R$ (cf. Proposition \ref{zimnonergo}), and fix a graphing $\Psi_0$ of C-cost $1$ which generates $\mathcal R_0^Y$. Let $f$ be an $M_{\mathcal R}$-measurable function, everywhere positive. Then there exists a graphing $\Phi$ whose C-cost is everywhere less than $\CCost(\mathcal R)-1+f$, such that $\Phi_0\cup\Phi$ generates $\mathcal R$.
\end{lem}
\begin{proof}
Begin with a graphing $\Phi$ which generates $\mathcal R$, and whose C-Cost is less than $f/2+\CCost(\mathcal R)$.
By Proposition \ref{maha}, we can find $A\subseteq X$ whose conditional measure is equal to $f/3$. Because $f$ is everywhere positive, $A$ is a complete section for $\mathcal R$, meaning that it meets the $\mathcal R$-class of almost every $x\in X$. Now the induction procedure as described in \cite[lemme II.8]{MR1728876} yields a treeing $\Phi_0$ of $\mathcal R_0$ whose C-cost is $(1-\mu_{\mathcal R}(A))$ and a graphing $\tilde\Phi$ of $\mathcal R$ whose C-cost is $(\CCost(\Phi)-(1-\mu_{\mathcal R}(A)))$ such that $\Phi_0$ and $\tilde\Phi$ generate $\mathcal R$. Then in particular $\Psi_0\cup \tilde\Phi$ generates $\mathcal R$. 
\end{proof}

The relation between the C-cost and the cost of the ergodic components is very simple, as shown by the following proposition.

\begin{prop}Let $\mathcal R$ be a pmp equivalence relation. Then the conditionnal cost of $\mathcal R$ is the function which associates to $x\in X$ the cost of $\mathcal R$ for the $\mathcal R$-ergodic measure whose support contains $x$. 
\end{prop}
\begin{proof}
This is a simple reformulation of the proof of Proposition 18.4.  in \cite{MR2095154}. \end{proof}

\section{The hyperfinite ergodic equivalence relation \texorpdfstring{$\mathcal R_0$}{R0}} \label{secro}

Let $2^\N=\{0,1\}^\N$, and for $n\in\N$, $2^n=\{0,1\}^n$. Given $s\in 2^n$, we define the basic clopen set $$N_s=\{x\in 2^\N: x_i=s_i\text{ for } 1\leq i\leq n\}.$$
We can see elements $a\in 2^n$ and $b\in2^\N\cup\bigcup_{n\in\N}2^n$ as words in $\{0,1\}$, and denote their concatenation by $a\smallfrown b$. For $\epsilon\in\{0,1\}$ and $n\in\N$,  $\epsilon^n$ is the word $(x_i)_{i=1}^n\in 2^n$ defined by $x_i=\epsilon$.

\subsection{Dyadic permutations}\label{secdyad}

Let $n\in\N$. We view $\mathfrak S_{2^n}$ as the group of permutations of the set $2^n$. This defines a natural inclusion $\alpha_n:\mathfrak S_{2^n}\into\mathfrak S_{2^{n+1}}$ given by 
$$\alpha_n(\sigma)(x_1,...,x_{n+1})=(\sigma(x_1,...,x_n),x_{n+1})$$
for $\sigma\in\mathfrak S_{2^n}$ and $(x_1,...,x_{n+1})\in2^{n+1}$. Let $\mathfrak S_{2^\infty}$ be the inductive limit of these groups, called the group of \textbf{dyadic permutations}.

Let us now define a function $\sqrt\cdot:\mathfrak S_{2^p}\to \mathfrak S_{2^{p+1}}$ such that for any $\sigma\in \mathfrak S_{2^p}$, $\sigma=(\sqrt \sigma)^2$ and $\sqrt \sigma$ has same support as $\sigma$ (seeing $\mathfrak S_{2^p}$ as a subgroup of $\mathfrak S_{2^{p+1}}$). The idea is just to make $\sqrt \sigma$ ‘‘twice as slow'', as  Figure \ref{figsqrt} shows.

\begin{figure}[htbp]\label{figsqrt}
\centering
\begin{tikzpicture}[scale=2]
	\begin{pgfonlayer}{nodelayer}
	\node [style=none] (N0) at (-1, 1.25) {$0$};
	\node [style=none] (N1) at (1, 1.25) {$1$};
	\node [style=none] (N00) at (-1.5, -1.25) {$00$};
	\node [style=none] (N01) at (-0.5, -1.25) {$01$};
	\node [style=none] (N10) at (0.5, -1.25) {$10$};
	\node [style=none] (N00) at (1.5, -1.25) {$11$};	
		\node [style=none] (0) at (-2, 1) {};
		\node [style=none] (1) at (-2, 0) {};
		\node [style=none] (2) at (0, 1) {};
		\node [style=none] (3) at (0, 0) {};
		\node [style=none] (4) at (2, 0) {};
		\node [style=none] (5) at (2, 1) {};
		\node [style=none] (6) at (1, 0) {};
		\node [style=none] (7) at (1, -1) {};
		\node [style=none] (8) at (0, -1) {};
		\node [style=none] (9) at (2, -1) {};
		\node [style=none] (10) at (-1, -1) {};
		\node [style=none] (11) at (-2, -1) {};
		\node [style=none] (12) at (-1, 0) {};
		\node [style=new] (13) at (-1.5, -0.5) {};
		\node [style=new] (14) at (-0.5, -0.5) {};
		\node [style=new] (15) at (0.5, -0.5) {};
		\node [style=new] (16) at (1.5, -0.5) {};
		\node [style=new2] (17) at (-1, 0.5) {};
		\node [style=new2] (18) at (1, 0.5) {};

	\end{pgfonlayer}
	\begin{pgfonlayer}{edgelayer}
		\draw (0.center) to (5.center);
		\draw (5.center) to (9.center);
		\draw (9.center) to (11.center);
		\draw (11.center) to (0.center);
		\draw (1.center) to (4.center);
		\draw (6.center) to (7.center);
		\draw [in=-90, out=90] (8.center) to (2.center);
		\draw (12.center) to (10.center);

		\draw [bigtiparrow2,bend left=20] (17) to (18);
		\draw [bigtiparrow2,bend left=20] (18) to (17);		
		\draw [bigtiparrow2,bend left=45] (14) to (15);
		\draw [bigtiparrow2,bend left=45] (15) to (16);
		\draw [bigtiparrow2, bend left=45] (13) to (14);
		\draw [bigtiparrow2,bend left=20] (16) to (13);
				\node [style=new] (19) at (0, 0.5) {$\sigma$};
		\node [style=new] (20) at (0, -0.5) {$\sqrt \sigma$};
	\end{pgfonlayer}
\end{tikzpicture}
\caption{A transposition $\sigma\in\mathfrak S_2$ and its square root $\sqrt \sigma\in \mathfrak S_4$.}
\end{figure}
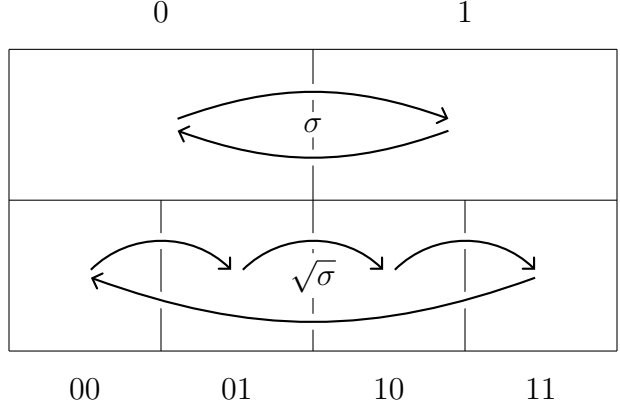

To be more precise, if $A\subseteq 2^p$ is the support of $\sigma$, define $\sqrt \sigma$ by $$\sqrt \sigma(s\smallfrown\epsilon)=\left\{\begin{array}{ll}s\smallfrown 1 & \text{if } s\in A\text{ and }\epsilon=0 \\\sigma(s)\smallfrown 0 & \text{if }s\in A\text{ and }\epsilon=1  \\s\smallfrown \epsilon & \text{if }s\not\in A\end{array}\right..$$It is easily checked that $\sqrt \sigma^2=\sigma$, and that $\sqrt \sigma$ has the same support as $\sigma$. 

This allows us to define functions $\sqrt[2^q]{\cdot}: \mathfrak S_{2^p}\to \mathfrak S_{2^{p+q}}$ inductively on $q\in\N$  by $\sqrt[2^{q+1}]T=\sqrt{\sqrt[2^q]T}$. 
\\

The key feature of $\mathfrak S_{2^\infty}$ is that it acts in a measure-preserving way on $(2^\N,\lambda)$ by, for $\sigma\in\mathfrak S_{2^n}$, $s\in 2^n$ and $x\in 2^\N$,
$$\sigma(s\smallfrown x)=\sigma(s)\smallfrown x.$$
Let $\mathcal R_0$ be the measure-preserving equivalence relation generated by this countable group. It is immediate that
$$(x_i)_{i\in\N}\,\mathcal R_0\,(y_i)_{i\in\N}\iff\exists n\in\N\text{ such that }\forall i\geq n, x_i=y_i.$$
To avoid confusion, when we see $\mathfrak S_{2^n}$ as a subgroup of $[\mathcal R_0]$, we will denote it by $\Stild_{2^n}$. We still have square root functions  $\sqrt[2^q]{\cdot}: \Stild_{2^p}\to \Stild_{2^{p+q}}$.\\

The following proposition belongs to the folklore, for a proof see e.g. \cite[Prop. 3.8]{MR2583950}.
\begin{prop}\label{permdense}The group of dyadic permutations is dense in the full group of $\mathcal R_0$.
\end{prop}
\subsection{The odometer \texorpdfstring{$T_0$}{T0}}

The \textbf{odometer} is the map $T_0\in\Aut(2^\N,\lambda)$ defined by
$$(x_i)_{i\in\N}\in 2^\N\mapsto 0^{n-1}1\smallfrown (x_i)_{i>n},$$
where $n$ is the first integer such that $x_n=0$ (note that this is well defined on a set of full measure). This can be understood as adding $(1,0,0,...)$ to $(x_i)_{i\in\N}$ with right carry. 

The map $T_0$ belongs to $\Aut(2^\N,\lambda)$,  an explicit inverse being given by 
$$T_0\inv: (x_i)\mapsto 1^{n-1}0\smallfrown (x_i)_{i>n}$$
where $n$ is the first integer such that $x_n=1$. This inverse can of course be understood as substraction with  right borrow. Moreover, one can check that $T_0$ generates $\mathcal R_0$.

Let $n\in\N$, then we define a ‘‘finite odometer'' $\sigma_n\in\mathfrak S_{2^n}$ by 
$$\sigma_n((s_i)_{i=1}^n)=\left\{\begin{array}{ll}0^n &\text{if }(s_i)=1^n  \\0^{k-1}1\smallfrown (s_i)_{i>k} & \text{else, where }k \text{ is the first integer such that }s_k=0. \end{array}\right.
$$
We denote by $T_n$ the corresponding element in $\Stild_{2^n}$.
Note that by definition, $T_n$ and $T_0$ coincide on $2^\N\setminus N_{1^n}$.

\section{Two topological generators for the full group of \texorpdfstring{$\mathcal R_0^Y=\Delta_Y\times\mathcal R_0$}{R0Y}}

In this section, our aim is to prove the following theorem, which gives two topological generators for non necessarily ergodic aperiodic hyperfinite equivalence relations, the second topological generator having an arbitrarily small support.

\begin{thm}\label{thmgenEO}
Let $(Y,\nu)$ be a standard probability space, possibly with atoms, and let $(X,\mu)=(Y\times 2^\N, \nu\times\lambda)$ where $\lambda$ is the standard Bernoulli measure.. Let $(Y_k)_{k\in\N}$ be a partition of $Y$. Finally, let $\mathcal R_0^Y=\Delta_Y\times\mathcal R_0$ be the pmp equivalence relation on $(X,\mu)=(Y\times 2^\N, \nu\times\lambda)$ defined by
$$(x,x')\,\mathcal R_0^Y\,(y,y')\iff x=y\text{ and }x'\,\mathcal R_0\,y'.$$

Then there exists  $U\in[\mathcal R_0^Y]$ whose support is contained in $Z=\bigsqcup_{k\in\N} Y_k\times (N_{0^{k+1}}\cup N_{1^{k+1}})$ such that  the following properties are satisfied for all $m\geq 0$:

\begin{enumerate}[(I)]\item\label{Ueng} The full group of $\mathcal R_0^Y$ is topologically generated by $T=\mathrm{id}_Y\times T_0$ and  $U^{2^m}$, where $T_0$ is the odometer.
\item\label{CUeng} For every odd cycle $C$ with disjoint support from $U^{2^m}$, the closed group generated by $(CU)^{2^m}$ contains $U^{2^m}$
\end{enumerate}
\end{thm}

The following two sections are devoted to basic facts about the full group of $\mathcal R$, which will be used in the proof of Theorem \ref{thmgenEO} that we give afterwards.

\subsection{Some subgroups which topologically generate \texorpdfstring{$[\mathcal R_0^Y]$}{[R0Y]}}
\label{subtopgen}
Let $\mathcal R_0^Y$ be as in the statement of the previous theorem. By definition, the full group of $\mathcal R_0^Y$ is isomorphic to the group of measurable maps from $(Y,\nu)$ to $[\mathcal R_0]$, denoted by $\LL^0(Y,\nu,[\mathcal R_0])$. This identification is isometric when we put on $\LL^0(Y,\nu,[\mathcal R_0])$ the $\LL^1$ metric $d_1$ defined by
$$\forall f,g: Y\to[\mathcal R_0],\;\;\;\;d_1(f,g)=\int_Yd_u(f(y),g(y))d\nu(y).$$
Now given a non-null subset $A$ of $Y$, we can define an injective group morphism $\iota_A:[\mathcal R_0]\to [\mathcal R_0^Y]$ by letting $\iota_A(V)$ be the element of $\LL^0(Y,\nu,[\mathcal R_0])$ which sends every $y\in A$ to $V$, and every $y\not\in A$ to $\mathrm{id_{2^\N}}$. This morphism is bi-Lipschitz, and if $A, B$ are disjoint subsets of $Y$ and $V,W\in[\mathcal R_0]$, the (commuting) product $i_A(V)i_B(W)$ can be understood as the map

$$y\in Y\mapsto\left\{\begin{array}{cl}V & \text{ if }y\in A, \\W & \text{ if }y\in B, \\\mathrm{id}_{2^\N} & \text{ otherwise.}\end{array}\right..$$

Since $[\mathcal R_0]$ is separable, the measurable maps from $(Y,\nu)$ to $[\mathcal R_0]$ with countable range are dense in $\LL^0(Y,\nu,[\mathcal R_0])$, which in turn yields that the measurable maps with finite range are dense in $\LL^0(Y,\nu,[\mathcal R_0])$. But every such map $f$ can be decomposed as a product

$$ f= \iota_{A_1}(V_1)\cdots\iota_{A_n}(V_n)$$
where $V_i\in [\mathcal R_0]$ and the $A_i$'s partition $Y$.  Suppose now that $\mathcal B$ is a dense  subalgebra of $\MAlg(Y,\nu)$, we may then find $\tilde f$ arbitrarily close to $f$, where 
$$\tilde f= \iota_{B_1}(V_1)\cdots\iota_{B_n}(V_n)$$
and the $B_i$'s belong to $\mathcal B$. We have proved the following lemma:

\begin{lem}\label{iotadense}The full group of $\mathcal R_0^Y$ is topologically generated by the set $$\bigcup_{B\in\mathcal B}\iota_B([\mathcal R_0]),$$
where $\mathcal B$ is any dense subalgebra of $\MAlg(Y,\nu)$.
\end{lem}

\subsection{Two generators for many finite subgroups}\label{gentun}

Let $n\geq 2$, and define $\tau_n\in\mathfrak S_{2^n}$ to be the transposition which exchanges $0^{n-1}1$ and $1^{n-1}0$. Let $U_n$ be the corresponding element of $\Stild_{2^n}$, that is, the group element implementing the action of $\tau_n$ on $2^\N$. Note that the support of $U_n$ is $N_{0^{n-1}1}\cup N_{1^{n-1}0}$, so that the supports of the $U_n$'s are all disjoint.

\begin{lem}\label{generiota}For all $A\subseteq Y$ non-null, $\iota_A(\Stild_{2^n})$ is contained in the group generated by $\iota_A(U_n)$ and $T$. Moreover, there exists $\kappa(n)\in\N$, independent of $A$, such that every element of $\iota_A(\Stild_{2^n})$ can be written as a word in $T$ and $\iota_A(U_n)$ of size less than $\kappa(n)$.
\end{lem}
\begin{proof}Recall that $T_n\in\Stild_{2^n}$ is obtained from the ‘‘finite odometer'' $\sigma_n\in \mathfrak S_{2^n}$, which acts transitively on $2^n$.  And $U_n$ corresponds to a transposition $\tau_n$ which permutes $1^{n-1}0$ and $0^{n-1}1=\sigma_n(1^{n-1}0)$. For $i\in\{1,...,2^n\}$, let $a_i=\sigma_n^{i-1}(0^n)$, and note that $a_{2^n}=1^n$. For $i\in\{1,...,2^n-1\}$, let $\rho_i$ be the transposition exchanging $a_i$ and $a_{i+1}$, and let $V_i$ be the corresponding element of $\Stild_{2^n}$. Note that  by construction the $\rho_i$'s generate $\mathfrak S_{2^n}$, so that the $V_i$'s generate $\Stild_{2^n}$. Also, since $a_{2^{n-1}}=1^{n-1}0$, we have $\rho_{2^{n-1}}=\tau_n$. 

Then observe that, for $j\in\{-2^{n-1}+1,...,2^{n-1}-1\}$, we have the conjugation relation
$$\sigma_n^j\rho_{2^{n-1}}\sigma_n^{-j}=\rho_{2^{n-1}+j}.$$
Going back to $[\mathcal R_0]$, this yields
$$T_n^jU_nT_n^{-j}=V_{2^{n-1}+j}.$$
But because $T$ and $\iota_A(T_n)$ coincide on $A\times (2^\N\setminus N_{1^n})$, and $\iota_A(U_n), \iota_A(V_{2^{n-1}+j})$ act trivially outside of $A\times 2^\N$, the previous formula yields
$$T^j\iota_A(U_n)T^{-j}=\iota_A(V_{2^{n-1}+j}).$$
So the group generated by $T$ and $\iota_A(U_n)$ contains all the $\iota_A(V_i)$ for $1\leq i\leq 2^n-1$. Since the $V_i$'s generate $\Stild_{2^n}$, the group generated by $T$ and $\iota_A(U_n)$ contains $\iota_A(\Stild_{2^n})$. And the second part of the lemma follows from the finiteness of $\Stild_{2^n}$ and the fact that the construction is independent from the subset $A$ of $Y$ we chose.
\end{proof}

We now fix once and for all a function $n\mapsto \kappa(n)$ given by the previous lemma.

\subsection{Two topological generators for \texorpdfstring{$[\mathcal R_0^Y]$}{[R0Y]}}

We now begin the proof of Theorem \ref{thmgenEO}. Let $(X,\mu)=(Y\times 2^\N,\nu\times\lambda)$ and $\mathcal R_0^Y=\Delta_Y\times\mathcal R_0$ be the equivalence relation $(x,y)\,\mathcal R_0^Y\,(x',y')$ iff $x=x'$ and $y\,\mathcal R_0\,y'$. Note that by Dye's theorem, every aperiodic singly generated equivalence relation is orbit equivalent to such an $\mathcal R_0^Y$.

We also define $T=\mathrm{id}_Y\times T_0\in[\mathcal R_0^Y]$, where $T_0$ is the odometer.

We first focus on property (\ref{Ueng}) in the case $m=0$, and we will see that the constructed $U$ also satisfies properties (\ref{Ueng}) and (\ref{CUeng}) for all $m\in\N$. So we fix a partition $(Y_k)_{k\in\N}$ of $Y$, and we want to find $U$ such that its support is a subset of
$$Z=\bigsqcup_{k\in\N} Y_k\times (N_{0^{k+1}}\cup N_{1^{k+1}}),$$
and the set $\{T,U\}$ topologically generates $[\mathcal R_0^Y]$.
Note that the conditional measure of $Z$ on $\LL^2(Y,\nu)$ is an arbitrary small positive function. To be more precise, if $f\in \LL^2(Y,\nu)$ is any function such that $1\geq f(y)>0$ for all $y\in Y$, define $Y_k=\{y\in Y: \frac 1{2^k}\leq f(y)<\frac 1{2^{k-1}}\}$, and note that the conditional measure of $Z=\bigsqcup_{k\in\N} Y_k\times (N_{0^{k+1}}\cup N_{1^{k+1}})$ is smaller than $f$.

Let $\mathcal A_n$ be an increasing family of finite subalgebras of $\MAlg(Y,\nu)$ with a dense reunion. Define $\mathcal B_n=\{A\cap (\bigcup_{k\leq n} Y_k): A\in\mathcal A_n\}$ and let $\mathcal B=\bigcup_{n\in\N} \mathcal B_n$. Note that because $\nu(\bigcup_{k\leq n} Y_k)\to 1$ as $n\to+\infty$, the countable set $\mathcal B$ is dense in $\MAlg(Y,\nu)$, and we have the following additionnal property: for every $B\in\mathcal B$, there exists $K\in\N$ such that $B\times (N_{0^K}\cup N_{1^K})\subseteq Z$.

First recall that for all $B\in\mathcal B$, there exists $K\in\N$ such that $B\times (N_{0^K}\cup N_{1^K})\subseteq Z$. Since the support of $\iota_B( U_n)$ is $B\times N_{0^{n-1}1}\cup N_{1^{n-1}0}$, which is a subset of $B\times (N_{0^K}\cup N_{1^K})$ for all $n\geq K-1$, we have the following lemma.
\begin{lem}\label{suppdisj}For every $B\in\mathcal B$ and $n$ large enough, the support of $\iota_B(U_n)$ is a subset of $Z$.
\end{lem}

We now build an element $U$ of the full group of $\mathcal R$ such that the support of $U$ is a subset of $Z$ and $[\mathcal R_0^Y]$ is topologically generated by $T$ and $U$.

\begin{proof}[Sketch of the construction]Let us briefly explain the idea behind the construction of $U$. Observe that for all $B,B'\in\mathcal B$ and $n\neq m\in\N$, $\iota_B(U_n)$ and $\iota_{B'}(U_m)$ have disjoint support (hence they commute). We fix an enumeration $(B_k)_{k\in\N}$ of $\mathcal B$ such that each element of $\mathcal B$ appears infinitely many times. Now start with $U=\iota_{B_1}(U_{n_1})$, where $n_1$ is such that the support of $\iota_{B_1}(U_{n_1})$ is a subset of $Z$. Then by Lemma \ref{generiota}, the group generated by $T$ and $U$ contains $\iota_{B_1}(\Stild_{2^{n_1}})$. 

We now also want the group generated by $T$ and $U$ to contain $\iota_{B_2}(\Stild_{2^{n_2}})$ for some $n_2$ such that the support of $\iota_{B_1}(U_{n_2})$ is a subset of $Z$. If we put $U=\iota_{B_1}(U_{n_1})\iota_{B_2}(\sqrt{U_{n_2}})$, observe that since $U_{n_1}$ is an involution, $U^2=\iota_{B_2}(U_{n_2})$ so that the group generated by $T$ and $U$ contains $\iota_{B_2}(\Stild_{2^{n_2}})$. The problem is that now, the group generated by $T$ and $U$ does not contain $\iota_{B_1}(\Stild_{2^{n_1}})$ anymore. But if we choose $n_2$ big enough, the support of $\iota_{B_2}(\sqrt U_{n_2})$ will be very small, so we will still have $\iota_{B_1}(\Stild_{2^{n_1}})$ in the group generated by $T$ and $U$, up to a small error. If we keep doing that, the group generated by $T$ and $U$ will contain up to an error converging to $0$ every $\iota_{B_k}(\Stild_{2^{n_k}})$, so by Proposition \ref{permdense} it will actually contain every $\iota_B([\mathcal R_0])$. Then lemma \ref{iotadense} yields that the group topologically generated by $T$ and $U$ contains the full group of $\mathcal R_0^Y$. 
\end{proof}

We now begin the actual construction, reverse-engineering the sketch above. Thanks to Lemma \ref{iotadense}, it is enough to find $U\in[\mathcal R_0^Y]$ such that the support of $U$ is a subset of $Z$, and for every $B\in\mathcal B$, $\iota_B([\mathcal R_0])$ is contained in the closed group generated by $T$ and $U$.

Fix a sequence of positive reals $(\epsilon_k)_{k\in\N}$ decreasing to zero. Since $\iota_B$ is bi-Lipschitz, by Proposition \ref{iotadense} it suffices to find $U\in[\mathcal R_0^Y]$ whose support is a subset of $Z$ and an increasing sequence of integers $(n_k)_{k\in\N}$ such that
\begin{enumerate}[(a)]
\item\label{conda}For all $B\in\mathcal B$ and infinitely many $k\in\N$, every element of $\iota_B(\Stild_{2^{n_k}})$ is in the $\epsilon_{k}$-neighborhood of $\la T,U\ra$.\end{enumerate}

We now fix an enumeration $(B_k)_{k\in\N}$ of $\mathcal B$ such that each element of $\mathcal B$ appears infinitely many times. This allows us to replace condition (\ref{conda}) by condition

\begin{enumerate}[(a')]
\item\label{condaprime} For all $k\in\N$, every element of $\iota_{B_k}(\Stild_{2^{n_k}})$ is in the $\epsilon_{k}$-neighborhood of $\la T, U\ra$.\end{enumerate}

Finally, the definition of $\kappa(n)$ (cf. the end of section \ref{gentun}), the fact that $d_u$ is biinvariant and Lemma \ref{generiota} yield that we can replace (\ref{condaprime}') by the weaker condition

\begin{enumerate}[(a'')]
\item\label{condaseconde} For all $k\in\N$, the element $\iota_{B_k}(U_{n_k})$ is in the $\displaystyle\frac{\epsilon_{k}}{\kappa(n_k)}$-neighborhood of the group generated by $U$.\end{enumerate}

First choose $n_1\in\N$ such that the support of $\iota_{B_1}(U_{n_1})$ is a subset  of $Z$ (such an $n_1$ exists by Lemma \ref{suppdisj}). Then, given $n_k$, find $n_{k+1}>n_k$ such that
\begin{enumerate}[(i)]
\item the support of $\iota_{B_{k+1}}(U_{n_{k+1}})$ is a subset of $Z$ and
\item \label{formula}$\displaystyle\frac 1{2^{n_{k+1}-2}}<\frac{\epsilon_{k}}{\kappa(p_{k})}$.
\end{enumerate}
We are ready to define $U$ as an infinite product of commuting elements:

$$U=\prod_{l\in\N^*}\iota_{B_l}(\sqrt[2^{l-1}] {U_{n_l}}).$$

Let us check that condition (\ref{condaseconde}'') is satisfied. If we fix $k\in\N^*$, we have

\begin{align*}
U^{2^{k}}&=\prod_{l\in\N^*}\iota_{B_l}(\sqrt[2^{l-1}] {U_{n_l}}^{2^{k-1}})\\
&=\prod_{l=1}^{k-1} \iota_{B_l}(U_{n_l}^{2^{k-l}})\cdot \iota_{B_k}(U_{n_k})\cdot\prod_{l=k+1}^{+\infty}\iota_{B_l}(\sqrt[2^{l-k}]{U_{n_l}})
\end{align*}
Because the $U_{n_l}$'s are involution, the first product is equal to the identity, so that
\begin{equation}\label{product}
U^{2^{k-1}}=\iota_{B_k}(U_{n_k})\cdot\prod_{l=k+1}^{+\infty}\iota_{B_l}(\sqrt[2^{l-k}]{U_{n_l}}).
\end{equation}
We now check that the error term $W_k=\prod_{l=k+1}^{+\infty}\iota_{B_l}(\sqrt[2^{l-k}]{U_{n_l}})$ is small. Because for every $l\in\N$, $\sqrt[2^{l-k}]{U_{n_l}}$ has same support as $U_{n_l}$, the support of $W_k$ has measure smaller than 
\begin{equation}\label{sum}\nu(B_l)\sum_{l=k+1}^{+\infty}\lambda(\supp U_{n_l})\leq\sum_{l=k+1}^{+\infty}\frac 1{2^{n_l-1}}\end{equation}
Since $(n_l)_{l\in\N}$ is increasing, we have for all $l\geq k+1$, $$\frac 1{2^{n_l-1}}\leq \frac 1{2^{n_{k+1}+(l-k-2)}}$$
We can now bound from above the sum in (\ref{sum}) and get the inequality 
\begin{align*}\mu(\supp W_k)&\leq\frac1{2^{n_{k+1}}}\cdot \sum_{l={k+1}}^{+\infty} \frac 1{2^{l-k-2}}\\
&\leq\frac 4{2^{n_{k+1}}}\\
&\leq \frac{\epsilon_k}{\kappa(n_k)}
\end{align*}
because condition (\ref{formula}) states that $\displaystyle\frac 1{2^{n_{k+1}-2}}<\frac{\epsilon_{k}}{\kappa(p_{k})}$.  In the end, because $d_u$ is bi-invariant, formula (\ref{product}) yields
\begin{align*}
d_u(U^{2^{k-1}},\iota_{B_k}(U_{n_k}))&\leq \mu(\supp(W_k))\\
&\leq\frac{\epsilon_k}{\kappa(n_k)}.
\end{align*}
Condition (\ref{condaseconde}'') is thus satisfied, so that $T$ and $U$ topologically generate the full group of $\mathcal R_0^Y$. In other words, property (\ref{Ueng}) holds for $m=0$. Let now $m\geq 0$, and let $U'=U^{2^m}$. By the definition of the square root function, and the fact that the $U_n$'s are involutions, we get  
\begin{align*}U'&=\prod_{l=m+1}^{+\infty}\iota_{B_l}(\sqrt[2^{l-m-1}]{U_{n_l}})\\
&=\prod_{l\in\N^*}\iota_{B_{m+l}}(\sqrt[2^{l-1}]{U_{n_{l+m}}}).\end{align*}
So if we put $B'_l=B_{m+l}$, the $B'_l$'s still enumerate $\mathcal B$ in a way that every element of $\mathcal B$ appears infinitely many times, and we can apply the previous proof to see that $U'$ satisfies condition
\begin{enumerate}[(b'')]
\item For all $k\in\N^*$, the element $\iota_{B'_k}(U_{n_{k+m}})$ is in the $\displaystyle\frac{\epsilon_{k+m}}{\kappa(n_{k+m})}$-neighborhood of the group generated by $U'$.\end{enumerate}
Using the same argument as for condition (a'') on $U$, one can see that this implies that $T$ and $U'$ topologically generate the full group of $\mathcal R_0^Y$. So for all $m\geq 0$, we have that $T$ and $U^{2^m}$ topologically generate the full group of $\mathcal R_0^Y$, in other words condition (\ref{Ueng}) is satisified.

Let us now check that condition (\ref{CUeng}) is also satisfied for all $m\geq 0$, that is, let us show that if $C$ is an odd cycle (an element whose orbits are all finite and have odd cardinality) with disjoint support from the support of $U^{2^m}$, then $U^{2^m}$ is a cluster point of the sequence $\left((CU^{2^m})^k\right)_{k\in\N}$. For this, we just need the fact that $V=U^{2^m}$ only has orbits of cardinality $2^l$ for $l\in\N$.

Let the integers $(p_l)_{l\in\N}$ enumerate the cardinalities of the orbits of $C$, and let $P_k=\prod_{l\leq k} p_k$. We have $C^{P_k}\to \mathrm{id}_X$ when $k\to+\infty$. For $k\in\N$, let $X_k$ be the set of elements whose $V$-orbit has cardinality at most $2^k$. Then, because all the orbits of $V$ are finite, we have $\mu(X_k)\to 1$ as $k\to+\infty$. So, if we let 
$$V_k(x)=\left\{\begin{array}{cl}V(x) & \text{if }x\in X_k \\x &\text{else}\end{array}\right.,$$
we can write $V=V_kW_k$ , where $V_k$ and $W_k$ have disjoint support and $W_k\to \mathrm{id}_X$ as $k\to+\infty$.

Since $P_k$ is odd we find $m_k\in\N$ such that $m_kP_k\equiv 1 \mod 2^k$. Observe that then, $m_kP_k\equiv 1 \mod 2^l$ for all $l\leq k$. This implies that ${V}_k^{m_kP_k}=V_k$, so that 
\begin{align*}
d_u(V^{m_kP_k}, V_k)&=d_u({V}_k^{m_kP_k}{W}_k^{m_kP_k},{V}_k)\\
&=d_u({V}_k{W}_k^{m_kP_k},{V}_k)\\
&\leq\mu(\supp {W}_k)\to 0\;\; [k\to+\infty].\end{align*}
In the end, because $V$ and $C$ have disjoint support, 
\begin{align*}
(CV)^{m_kP_k}=(C^{P_k})^{m_k}V^{m_kP_k}\to V\;\; [k\to+\infty],
\end{align*}
so that condition (2) is satisfied and Theorem \ref{thmgenEO} is proved.

\section{Consequences}

\subsection{Computation of the topological rank for full groups}

We now begin the proof of Theorem \ref{thethm}. Let $\mathcal R$ be a pmp aperiodic equivalence relation. We first prove the inequality
$$t([\mathcal R])\geq\sup_{x\in X} \lfloor \CCost(\mathcal R)(x)\rfloor+1.$$
Let $A$ be a positive $\mathcal R$-invariant set and $n\in\N$ such that $\lfloor \CCost(\mathcal R)(x)\rfloor\geq n$ for all $x\in A$. Then the restriction to $A$ yields a continuous morphism $[\mathcal R]\to[\mathcal R_{\restriction A}]$. The normalised cost of $\mathcal R_{\restriction A}$ is at least $n$, so that by an argument of Miller (cf. \cite[proof of Thm. 1]{gentopergo}), $t([\mathcal R_{\restriction A}]\geq n+1$. Because the restriction is a continuous surjective morphism, we get $t([\mathcal R])\geq n+1$. By definition of the essential supremum, the first inequality is proved.

We now prove the reverse inequality. Let us first put ourselves in the situation provided by Proposition \ref{zimnonergo} and assume that $X=Y\times 2^\N$, $\mu=\nu\times\lambda$ where $(Y,\nu)$ is a standard probability space possibly with atoms, and that $\mathcal R$ contains the hyperfinite equivalence relation $\mathcal R_0^Y=\Delta_{Y}\times\mathcal R_0$. The $\mathcal R$-conditional expectation is then the projection onto $\LL^2(Y,\nu)$.

Let $n$ be the essential supremum of the function $y\mapsto\lfloor \CCost(\mathcal R)\rfloor(y)$. We must find $n+1$ topological generators for $[\mathcal R]$. By definition, the C-cost of $\mathcal R$ is everywhere less than $n+1$. 

Fix a graphing $\Phi_0$ which generates $\mathcal R_0^Y$, and whose C-cost is 1. By Lemma \ref{troispointcinq}, we can find a graphing $\Phi$ of $\mathcal R$ such that $\Phi_0\cup \Phi$ generates $\mathcal R$ and the C-Cost of $\Phi$ is everywhere lesser than $n$. Now define $f(y)=\frac{\CCost(\Phi)(y)}{n}<1$, and write $\Phi=\Phi_1\cup\cdots\cup\Phi_n$, where each $\Phi_i$ has C-cost $f$. We may find an odd valued function $q(y)$ such that for all $y\in Y$, 
$$\frac{q(y)+2}{q(y)}f(y)<1.$$

Theorem \ref{thmgenEO} then provides elements $T$ and $U$ of the full group of $\mathcal R_0^Y$ such that  the support of $U$ has $\mathcal R$-conditional measure less than $1-\frac{q(y)+2}{q(y)}f(y)$, and the following properties are satisfied:

\begin{enumerate}[(1)]\item The full group of $\mathcal R_0^Y$ is topologically generated by $T=\mathrm{id}_Y\times T_0$ and  $U$.
\item For every odd cycle $C$ with disjoint support from $U$, the closed group generated by $UC$ contains $U$.
\end{enumerate}  

We now let $Y_p=\{y\in Y: q(y)=p\}$ for $p\in\N$. Note that when $p$ is even, $Y_p$ is empty, and that the $Y_p$'s partition $Y$. Then by Lemma \ref{maha}, for every $p\in\N$ we may find disjoint subsets $A_p^1,...,A_p^{p+2}$ of  $(Y_p\times 2^\N)\setminus\supp U$, each of the same $\mathcal R$ conditional measure $\frac {f(y)}{p}$. 

Fix $p\in\N$ as well as $i\in\{1,...,n\}$, and consider the restriction of $\Phi_i$ to $Y_p\times 2^\N$. Using Proposition \ref{transitive}, we may assume, after cutting/gluing elements of $\Phi_i$ and pre/post-composing by partial isomorphisms of $\mathcal R_0^Y$, that the restriction of $\Phi_i$ to $Y_p\times 2^\N$ is a pre-$(p+1)$-cycle. 

Now for all $p\in\N$, fix $\psi_p\in[[\mathcal R_0^Y]]$ whose domain is $A_p^{p+1}$ and whose range is $A_p^{p+2}$. Put $\tilde\Phi_i=\Phi_i\cup\{\psi_p:p\in\N\}$. 

Then, $\mathcal R$ is still generated by $\Phi_0\cup\tilde\Phi_1\cup\cdots\cup\tilde\Phi_n$, and every $\tilde \Phi_i$ is a pre-$(p+2)$-cycle when restricted to $Y_p\times 2^\N$. Let $C_i^p$ be the corresponding $(p+2)$ cycles, and let $C_i$ be the odd cycle obtained by gluing together all the $C_i^p$. Finally, let $U_1=UC_1$.

\begin{claim}The $(n+1)$ elements $T,U_1,C_2,...,C_n$ topologically generate the full group of $\mathcal R$. 
\end{claim}
\begin{proof}[Proof of the claim]
Let $G$ be the closed group generated by $T,U_1,C_2,...,C_n$.
First, thanks to condition (2) and the fact that $U_1=UC_1$ belongs to $G$, we have that $U$ belongs to $G$. So by condition (1), $G$ contains $[\mathcal R_0^Y]$. It follows that $G$ contains $U\inv U_1=C_1$. 
Since every $\mathcal R_{\tilde \Phi_i}$ is generated by its restrictions to every $Y_p\times 2^\N$, which contain $\psi_p\in[[\mathcal R_0^Y]]$, the arguments in \cite[Prop. 9]{gentopergo} and the fact that $\mathcal R$ is the join of $\mathcal R_0^Y$, $\mathcal R_{\tilde\Phi_1}$,...,$\mathcal R_{\tilde\Phi_n}$ yield the conclusion.
\end{proof}

\subsection{A dense \texorpdfstring{$G_\delta$}{G delta} set of topological generators}

Let $APER$ denote the set of elements of $\Aut(X,\mu)$ whose orbits are all infinite. In this section we show the following proposition, which will be improved at the end of the next section (cf. Theorem \ref{Gdelta}).

\begin{prop}\label{littlegdelta}
Let $\mathcal R$ be any aperiodic pmp equivalence relation with $\Cost(\mathcal R)=1$. Then the set $$\{(T,U)\in (\mathrm{APER}\cap [\mathcal R])\times[\mathcal R]: \overline{\la T,U\ra}=[\mathcal R]\}$$
is a dense $G_\delta$ in $(\mathrm{APER}\cap[\RR])\times[\mathcal R]$.
\end{prop}
\begin{proof}
It is a standard fact that the set we are interested in is a $G_\delta$, so we only have to show its density. Because of Rohlin's lemma, any aperiodic element of $[\RR]$ has a dense conjugacy class in $[\RR]$, so that we actually only have to find one aperiodic $T\in[\RR]$ such that the set  $$\{U\in[\RR]: \overline{\la T,U\ra}=[\RR]\}$$
is dense in $[\mathcal R]$. 

By Proposition \ref{zimnonergo}, we can suppose $X=Y\times 2^\N$  equipped with a product measure $\mu=\nu\times\lambda$, where $(Y,\nu)$ is a standard probability space (possibly with atoms) and $\lambda$ is the standard $1/2$ Bernoulli product measure. Moreover, we may assume that $\mathcal R$ contains $\mathcal R_0^Y=\mathrm{id}_{Y}\times\mathcal R_0$, which is generated by $T=\mathrm{id}_Y\times T_0$. We will show that the set $\{U\in[\RR]: \overline{\la T,U\ra}=[\RR]\}$
is dense. 

Recall that the set of cycles is dense in the full group of $\mathcal R$. One can then show that the set of cycles whose support does not intersect $Y\times (N_{0^{k+1}}\cup N_{1^{k+1}})$ for some $k$, and whose orbits have bounded cardinality is dense in $[\mathcal R]$.

Fix such a cycle $C$ and $\epsilon>0$, write the order of $C$ as a product $2^KN$, where $N$ is odd. Also fix $k\in\N$ such that $Y\times (N_{0^{k+1}}\cup N_{1^{k+1}})$ is disjoint from the support of $C$, and $\mu(Y\times (N_{0^{k+1}}\cup N_{1^{k+1}}))<\epsilon$, and find an odd number $M$ coprime with $N$.

Theorem \ref{thmgenEO} provides us $U\in[\mathcal R_0^Y]$ whose support is a subset of $Z=Y\times (N_{0^{k+2}}\cup N_{1^{k+2}})$, such that the full group of $\mathcal R_0^Y$ is topologically generated by $T$ and $U^{2^K}$.

Now note that because $\mathcal R$ is aperiodic, its conditional cost is greater or equal than one. Since $\mathcal R$ has cost one, its conditional cost must then be constant equal to one.
Using the same ideas as in the proof of Theorem \ref{thethm}, we may find a cycle $C'$ of order $M$ supported on $Y\times (N_{0^{k+1}1})$ (which is disjoint from the supports of $U$ and $C$), such that each $C'$ orbit contains two consecutive $\mathcal R_0^Y$ related elements, and $\mathcal R$ is generated by $T_0$ and $C'$. We define $U'=UCC'$, note that $U'$ is $\epsilon$-close to $C$ so that the proof boils down to the following claim.

\begin{claim}$T$ and $U'$ topologically generate $[\mathcal R]$.
\end{claim}

Indeed, let $G$ the the closed subgroup generated by $T$ and $U'$. Since $CC'$ and $U$ have disjoint support and $CC'$ is an odd cycle, condition (\ref{CUeng}) in Theorem \ref{thmgenEO} ensures that $U^{2^K}$ belongs to the group topologically generated by $U'^{2^K}=(CC'U)^{2^K}$, so $G$ contains $[\mathcal R_0^Y]$. Then $G$ must contain $CC'=U'U\inv$, but by the chinese remainder theorem, because $C$ and $C'$ commute and their orders are coprime with each other, $G$ contains $C'$. Now because each $C'$ orbit contains two consecutive $\mathcal R_0^Y$ related elements, the same proof as for Theorem \ref{thmgenEO} yields that $G$ contains $[\mathcal R_{C'}]$, hence $G$ contains $[\mathcal R]$. 
\end{proof}

\section{Free subgroups in full groups}

This section is devoted to a generalization of the following theorem of Kechris, which we then use to find dense free subgroups in some full groups. Note that the set of tuples generating a free group is always a $G_\delta$ in a Polish group, so that in the next theorems, the key feature is the density of tuples generating a dense subgroup.

\begin{thm}[{\cite[Thm. 3.9]{MR2583950}}]Let $\mathcal R$ be a pmp ergodic equivalence relation and $n\in\N$. Then the set of $n$-tuples of aperiodic elements of $[\mathcal R]$ generating a free subgroup is a dense $G_\delta$ in the set of $n$-tuples of aperiodic elements of $[\mathcal R]$.
\end{thm}

By the Kuratowski-Ulam theorem, for $n=2$ the above statement is equivalent to having a dense $G_\delta$ of aperiodic elements $T$ for which there is a dense $G_\delta$ of aperiodic elements $U$ such that the subgroup generated by $T$ and $U$ is free. Our result yields a dense $G_\delta$ above \textit{any} $T\in[\mathcal R]$ of infinite order. 

\begin{thm}\label{freegp}Let $\mathcal R$ be a pmp equivalence relation, and let $T\in[\mathcal R]$ have infinite order. Then for all $n\in\N$ we have the following:\begin{enumerate}[(1)]
\item The set of $n$-uples $(U_1,...,U_n)$ of elements of $[\mathcal R]$ such that
$$\la T,U_1,...,U_n\ra \simeq \mathbb F_{n+1}$$
is a dense $G_\delta$ in $[\mathcal R]^n$.
\item If furthermore $\mathcal R$ is aperiodic, then the set of $n$-uples of aperiodic elements $(U_1,...,U_n)$  of $[\mathcal R]$ such that 
$$\la T,U_1,...,U_n\ra \simeq \mathbb F_{n+1}$$
is a dense $G_\delta$ in $([\mathcal R]\cap \mathrm{APER})^{n}$. 

\end{enumerate}
\end{thm}

\begin{proof}
For the sake of notational simplicity, we only give a detailed proof for $n=1$. Whenever $w$ is a reduced word in two letters $t$ and $u$, we denote by $w(\varphi,\psi)$ the evaluation of this word on elements $\varphi,\psi$ belonging to the pseudo-full group of $\mathcal R$. Note that $w(\varphi,\psi)$ may be nowhere defined, but that whenever $\varphi$ and $\psi$ belong to the full group of $[\mathcal R]$, $w(\varphi,\psi)$ also does. 

Let $T\in\mathcal R$ have infinite order, which is equivalent to asking that $T$ has unbounded orbits. We have to show that for every non-empty reduced word $w$ in $t$ and $u$, the set of $U\in[\mathcal R]$ (respectively in $U\in\mathrm{APER}\cap[\RR]$) such that $w(T,U)\neq 1$ is dense in $[\mathcal R]$ (respectively in $\mathrm{APER}\cap[\RR]$). Indeed, every one of these sets is open, and their intersection is the set of $U$'s such that $T$ and $U$ generate a free group.

So we fix a reduced word $w$, $\delta>0$ and $U\in[\RR]$ (respectively $U\in\mathrm{APER}\cap[\mathcal R]$). We want to find $U'\in [\RR]$ (respectively $U'\in\mathrm{APER}\cap[\mathcal R]$) such that $d_u(U,U')<\delta$ and $w(T,U')\neq\mathrm{id}_X$. Because $T$ has infinite order, we may restrict ourselves to the case when $w$ contains at least one occurrence of $u$. Up to conjugating $w$, we may then suppose that $w$ ends with the letter $u$ or $u\inv$. Let $n$ be the length of $w$.

Because $T$ has unbounded orbits, there is $A\subseteq X$ such that $(T^j(A))_{j=0}^{2(n+1)^2}$ is a disjoint family of subsets of $X$ . The map which associates to every $B\in\MAlg(X,\mu)$ the first return map $U_B$ induced by $U$ on $B$ is continuous. So, up to shrinking $A$, we may furthermore assume that 
\begin{equation*}
d_u(U,U_{X\setminus {\bigsqcup_{j=0}^{2(n+1)^2}T^j(A)}})<\frac \delta 2,
\end{equation*}
and also that $\mu(\bigsqcup_{j=0}^{2(n+1)^2}T^j(A))<\frac \delta 2$. 

For all $i\in\{0,...,n\}$ and $j\in\{-n,...,n\}$, let $A_{i,j}=T^{j+n+2(n+1)i}(A)$. We think of the $A_{i,j}$'s as squares on the plane whose center has coordinates $(i,j)$. In this picture, $T$ mostly acts by vertical translation.

Write $w=w_n\cdots w_1$, with $w_k\in\{t,t\inv,u,u\inv\}$. Let us define by induction on $k\in\{0,...,n\}$ a sequence of pairs of integers $(i_k,j_k)_{k=0}^n$ by putting $(i_0,j_0)=(0,0)$, and then
$$(i_{k+1},j_{k+1})=(i_k,j_k)+\left\{\begin{array}{ll}(0,1) & \text{ if }w_k=t \\(0,-1) & \text{ if }w_k=t\inv \\(1,0) & \text{ if }w_k=u^{±1}\end{array}\right..$$
Note that the sequence $(i_k,j_k)_{k=0}^n$ is injective because $w$ is a reduced word. We now build up an element $\psi$ of the pseudo full group of $\mathcal R$ such that $w(T,\psi)$ ‘‘follows the same path'' as the sequence $(i_k,j_k)$ (cf. Figure \ref{unfold}). 
\begin{figure}\begin{center}
\includegraphics[scale=0.9]{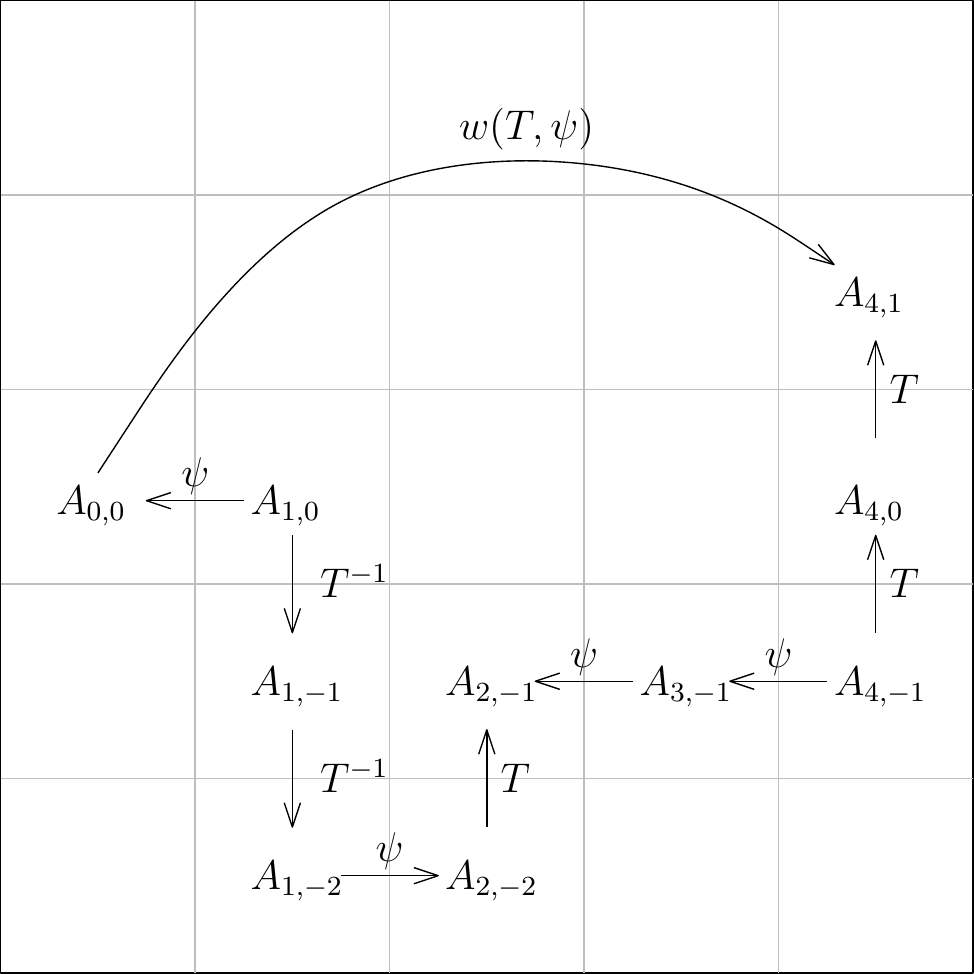}
\end{center}
\caption{\label{unfold}The construction of $\psi$ for $w=t^2u^{-2}tut^{-2}u^{-1}$}
\end{figure}

Let $K^+$ be the set of $k\in\{1,...,n\}$ such that $w_k=t$, and $K^-$ the set of $k\in\{1,...,n\}$ such that $w_k=t\inv$. For all $i,j\in\{1,...,n\}\times \{-n,...,n\}$, we fix a partial isomorphism $\psi_{i,j}\in[[\mathcal R]]$ whose domain is $A_{i-1,j}$ and whose image is $A_{i,j}$ (for instance, one can choose $\psi_{i,j}=T^{2n+1}_{\restriction A_{i-1,j}}$).

The partial isomorphism $\psi\in[[\mathcal R]]$ is defined by gluing together some of the $\psi_{i,j}$:
$$\psi=\bigsqcup_{k\in K^+} \psi_{i_k,j_k}\sqcup\bigsqcup_{k\in K^-}\psi_{i_k,j_k}\inv.$$

Note that $\psi$ is well defined because an element of $K^+$ is never followed by an element of $K^-$ and vice-versa, since $w$ is a reduced word. By construction, for all $k\in\{1,...,n\}$, $w_k\cdots w_1(T,\psi)$ has domain $A_{i_0,j_0}$ and range $A_{i_k,j_k}$. In particular, the range of $w_n\cdots w_1(T,\psi)=w(T,\psi)$ is disjoint from its domain.  

We now seek $U'\in [\RR]$ (respectively $U'\in\mathrm{APER}\cap[\RR]$) near to $U$ and which extends $\varphi$. First recall that we have
\begin{equation*}
d_u(U,U_{X\setminus {\bigsqcup_{j=0}^{2(n+1)^2}T^j(A)}})<\frac \delta 2,
\end{equation*}
and that the domain as well as the range of $\psi$ are subsets of $\bigsqcup_{j=0}^{2(n+1)^2}T^j(A)$. 

We now consider the two cases of the theorem separately.
In the first case, we extend $\psi$ to an element $\tilde U$ of the full group of the restriction of $\mathcal R$ to $\bigsqcup_{j=0}^{2(n+1)^2}T^j(A)$. In the second case, because $\psi$ is obtained by gluing together pre-$p$-cycles with disjoint supports, the aperiodicity of $\mathcal R$ enables us to extend it to an aperiodic element $\tilde U$ of the full group of the restriction of $\mathcal R$ to $\bigsqcup_{j=0}^{2(n+1)^2}T^j(A)$.

We may then define 
\begin{equation*}
U'(x)=\left\{\begin{array}{cl}\tilde U(x) & \text{ if }x\in \bigsqcup_{j=0}^{2(n+1)^2}T^j(A) \\ U_{X\setminus\bigsqcup_{j=-n}^nT^j(A)}(x) & \text{ else.}\end{array}\right.
\end{equation*}
We clearly have $d_u(U,U')<\delta$, and since $U'$ extends $\psi$, we hvae $w(T,U')\neq \mathrm{id}_X$. Finally, $U'$ belongs to $[\RR]$ (respectively to $\mathrm{APER}\cap[\RR]$).
\end{proof}

We now apply this theorem to give a characterization of aperiodic pmp equivalence relations of cost one, using the results of the previous section.

\begin{thm}\label{Gdelta}
Let $\mathcal R$ be any pmp equivalence relation, and consider the set $$G=\{(T,U)\in (\mathrm{APER}\cap [\mathcal R])\times[\mathcal R]: \overline{\la T,U\ra}=[\mathcal R]\text{ and } \la T,U\ra\simeq\mathbb F_2\}.$$
Then the following assertions are equivalent.
\begin{enumerate}[(1)]
\item $\mathcal R$ is aperiodic of cost one.
\item $G$ is a dense $G_\delta$ in $(\mathrm{APER}\cap [\mathcal R])\times[\mathcal R]$.
\end{enumerate}
\end{thm}
\begin{proof}
Suppose first that $\mathcal R$ is a aperiodic. Then Theorem \ref{freegp} yields that the set 
$$G_1=\{(T,U)\in (\mathrm{APER}\cap [\mathcal R])\times[\mathcal R]:\la T,U\ra\simeq\mathbb F_2\}$$
is a dense $G_\delta$, while Proposition \ref{littlegdelta} asserts that the set 
$$G_2=\{(T,U)\in (\mathrm{APER}\cap [\mathcal R])\times[\mathcal R]:\overline{\la T,U\ra}=[\mathcal R]\}$$
is a dense $G_\delta$. So $G=G_1\cap G_2$ is a dense $G_\delta$.

Conversely, if $\mathcal R$ is not aperiodic, it is not topologically finitely generated, hence $G_2$ is empty, so in particular $G$ is empty. And if $\mathcal R$ is aperiodic but not of cost one, its cost has to be strictly greater than one, so suppose by contradiction that $G_2$ were dense in $(\mathrm{APER}\cap [\mathcal R])\times[\mathcal R]$.  Its vertical section would contain $U$'s of arbitrarily small support, which along with any $T\in[\mathcal R]$ would then fail to generate the equivalence relation $\mathcal R$ by the definition of the cost, so that they could not topologically generate the full group of $\mathcal R$, a contradiction.
\end{proof}

\section{A finer metric on full groups}

The notion of $\mathcal R$-conditional measure induces a new metric $d_C$ on the full group of $\mathcal R$, defined by 
$$d_C(S,T)=\norm{\mu_{\mathcal R}(\supp S\inv T)}_{\infty}.$$
In this section, we study this bi-invariant complete metric, which is non-separable as soon as $\mathcal R$ has infinitely many ergodic components, and show that $([\mathcal R],d_C)$ satisfies the automatic continuity property as soon as $\mathcal R$ is aperiodic.

First, remark that if $\mathcal R$ has only finitely many ergodic components, this metric induces the uniform topology on $[\mathcal R]$, so that our automatic continuity result is a generalization of \cite[Thm. 3.1]{MR2599891}. Note however that our proof follows the exact same outline as theirs.

\subsection{Automatic continuity for \texorpdfstring{$([\mathcal R],d_C)$}{([R],dC)}}

The following notion, introduced by Rosendal and Solecki \cite{MR2365867}, will be key in order to prove the automatic continuity property for full groups.

\begin{df}Let $n\in\N$. A topological group $G$ is said to be $n$-\textbf{Steinhaus} if for every subset $A$ of $G$ containing the identity such that $A=A\inv$ and such that $G$ is covered by countably many left translates of $A$, the set $A^n$ of products of at most $n$ elements of $A$ contains a neighborhood of the identity. 
\end{df}

For a survey on automatic continuity, including in particular a proof of the next proposition, see \cite{MR2535429}. 

\begin{prop}[Rosendal-Solecki]Suppose a topological group $G$ is $n$-Steinhaus for some fixed $n\in\N$. Then every homomorphism from $G$ into a separable group is continuous.
\end{prop}

We may now state the main theorem of this section.

\begin{thm}\label{autocont}Let $\mathcal R$ be a pmp aperiodic equivalence relation. Then
$([\mathcal R], d_C)$ is 38-Steinhaus. In particular, every homomorphism from $([\mathcal R],d_C)$ into a separable group is continuous.
\end{thm}
\begin{proof}
Take $(B_n)_{n\in\N}$ which partition $X$, each of them having constant positive conditional measure. For any $B\subseteq X$, let $[\mathcal R]_B=\{T\in[\mathcal R]: \supp T\subseteq B\}$. The proof of the first step goes verbatim as in \cite{MR2599891}, so we omit it.
\begin{step}\label{firststep} There exists $n\in\N$ such that $$\forall T\in [\mathcal R]_{B_n}\exists S\in W^2, S_{\restriction B_n}=T_{\restriction B_n}$$
\end{step}
We fix $n$ as in the previous lemma, and let $B=B_n$.

\begin{step}\label{secondstep}$W^2$ contains an involution $S$ whose support is contained in $B$, and has constant conditional expectation lesser than $\mu(B)/2$.
\end{step}
Fix disjoint $A_1,A_2\subseteq B$ of constant conditional measure $\mu(B)/4$, and fix an involution $T$ supported in $A_1\cup A_2$, exchanging them. Now one can find an increasing family $(A_t)_{t\in(0,\mu(B)/4)\mapsto}$ of subsets of $A_1$ where each $A_t$ has constant conditional measure equal to $t$ (for instance,  Lemma \ref{maha} provides such a family for dyadic $t$, and one can then use the completeness of the measure algebra to extend this uniquely for $t\in (0,\mu(B)/4)$). 

Let $T_t=T_{\restriction A_t\cup T(A_t)}\sqcup \mathrm{id}_{X\setminus(A_t\cup T(A_t)}$, note that these involutions all belong to the commutative group $[\mathcal R_T]$, and as there are uncountably many such $T_t$'s, there is $m$ such that $g_mW$ contains two distincts elements $T_{t_1}, T_{t_2}$ with $t_1<t_2$. Now $S=T_{t_1}T_{t_2}=T_{t_1}\inv T_{t_2}$ belongs to $Wg_m\inv g_mW=W^2$ and has support $A_{t_2}\setminus A_{t_1}\sqcup T(A_{t_2}\setminus A_{t_1})$, which has constant conditional expectation lesser than $\mu(B)/2$. \\

We fix such an $S$, and a subset $C$ of $B$ such that $\supp S\subseteq C$ and $\mu_R(C)=2\mu_R(\supp S)$. 

\begin{step}We have $[\mathcal R]_C\subseteq W^{36}$.
\end{step}

Indeed, let $T\in [\mathcal R]_C$ be an involution. We first construct an involution $U\in[\mathcal R]_C$ such that the support of the involution $S(USU)$ has conditional measure $f=\mu_{\mathcal R}(\supp T)$. By Lemma \ref{maha}, we may find $E\subseteq (C\setminus\supp S)$ which has conditional measure $\frac{f}2$. Let $F_1$ be a subset of a fundamental domain of $S_{\restriction \supp S}$ having conditional measure $\frac{\mu(C)-f}4$. Put $F_2=T(F_1)$, let $F=F_1\sqcup F_2$ and let $\varphi\in[[\mathcal R]]$ be a partial isomorphism of domain $E$ and range $F$ given by Lemma \ref{transitive}. Define the involution $U\in[\mathcal R]_C$ by
$$U(x)=\left\{\begin{array}{cc}\varphi(x) & \text{if }x\in E \\\varphi\inv(x) & \text{if } x\in F \\x & \text{else}\end{array}\right.
$$
Then the support of $S(USU)$ is $E\cup F$, hence has conditional measure $f=\mu_{\mathcal R}(\supp T)$. Thanks to step \ref{firststep}, we may find $\tilde U\in W^2$ which coincides with $U$ on $\supp S\subseteq B$. Then, $\tilde US\tilde U\inv=USU$ so that $S(USU)$ belongs to $W^8$. Because its support has the same conditional measure as the support of $T$, Lemma \ref{transitive} yields that $SUSU$ and $T$ are conjugated by an involution $V$ supported on $C$. Again, by step \ref{firststep}, we find $\tilde V$ such that $V$ and $\tilde V$ coincide on $C$, and we deduce that $T=VSUSUV=\tilde V S\tilde U S\tilde U\inv\tilde V\inv$ belongs to $W^{12}$.

Now by a result of Ryzhikov  \cite{MR1244984} any element of $[\mathcal R]_C$ is the product of at most 3 involutions in $[\mathcal R]_C$, so $[\mathcal R]_C\subseteq W^{36}$.\\

We can now conclude the proof exactly as in \cite{MR2599891}: let $(T_n)$ be a sequence of elements of $[\mathcal R]$ such that $d_C(T_n,1)\to 0$. Let $D_n=\supp T_n$, and $D=\bigcup_m g_mD_m$. We only need to show that there exists $m$ such that $T_m\in W^{38}$. Going to a subsequence if necessary, we may assume that $\sum \mu_{\mathcal R}(D_n)\leq \mu_{\mathcal R}(C)$, hence $\mu_{\mathcal R}(D)\leq \mu_{\mathcal R}(C)$. 
Now there is $A\subseteq C$ such that $\mu_{\mathcal R}(A)=\mu_{\mathcal R}(D)$. We may find $S\in [\mathcal R]$ such that $S(A)=D$; there is $m$ such that $S\in g_m W$, and we put $T= g_m\inv S\in W$. We have $D_m\subseteq T(A)$, and so $T\inv T_mT\in [\mathcal R]_A\subseteq [\mathcal R]_C\subseteq W^{36}$. Thus $T_m$ belongs to $WW^{36}W=W^{38}$, and we are done.
\end{proof}

\subsection{Connectedness}

Whenever $\mathcal R$ is of type $\mathrm I_n$, the metric $d_C$ induces the discrete topology. However, if $\mathcal R$ is aperiodic, the induced topology is connected.

\begin{lem}\label{apiscon}Let $\mathcal R$ be an aperiodic pmp equivalence relation. Then $([\mathcal R],d_C)$ is connected.
\end{lem}
\begin{proof}
By a result of Ryzhikov  \cite{MR1244984}, $[\mathcal R]$ is generated by involutions, so it suffices to connect involutions to the identity. But if $T$ is an involution, let $A$ be a fundamental domain of $T$. Since $\mathcal R$ is aperiodic we may, as in step \ref{secondstep} of the proof of Theorem \ref{autocont}, find an increasing family $(A_t)_{t\in[0,1]}$ of subsets of $A$ such that $\mu_{\mathcal R}(A_t)=t\mu_{\mathcal R}(A)$. Then define the involution $T_t$ by
$$T_t(x)=\left\{\begin{array}{ll}T(x) & \text{if }x\in A_t \\T\inv(x) & \text{if } x\in T(A_t) \\x & \text{else }\end{array}\right.$$
Then $d_C(T_{t_1},T_{t_2})\leq\abs{t_1-t_2}$, $T_0=\mathrm{id}_X$ and $T_1=T$, so that $T$ is connected to the identity.
\end{proof}

Putting together Lemma \ref{apiscon} and Theorem \ref{autocont}, we have the following consequence.

\begin{cor}\label{nomorzerodim}Let $\mathcal R$ be a pmp aperiodic equivalence relation. Then every morphism from $[\mathcal R]$ to a totally disconnected separable group is trivial.
\end{cor}

\begin{rmq}This implies that the full group of an aperiodic pmp equivalence relation cannot act nontrivially by homeomorphisms on a Cantor set, in particular in can never be abstractly isomorphic to the full group of a Cantor dynamical system, answering a question suggested by Vincent Tassion (for a definition of full groups in the topological setting, cf. \cite{MR1710743}).
\end{rmq}

We can then characterize aperiodic pmp equivalence relations algebraically.

\begin{cor}\label{caraper}Let $\mathcal R$ be a pmp equivalence relation. Then $\mathcal R$ is aperiodic iff every morphism from its full group to a totally disconnected separable group is trivial. Moreover, when $\mathcal R$ is periodic, it admits a non-trivial morphism onto $\Z/2\Z$.
\end{cor}
\begin{proof}
By Corollary \ref{nomorzerodim}, we only need to show that if $\mathcal R$ is not aperiodic, then there exists a non-trivial morphism $[\mathcal R]\to \Z/2\Z$. So let $\mathcal R$ be a non-aperiodic equivalence relation, then we find an $\mathcal R$-invariant positive set $A$ and $n\in\N$ such that all the elements of $A$ have $\mathcal R$-classes of cardinality $n$. Then we have a morphism $\pi_A:[\mathcal R]\to[\mathcal R]_A$ given by restriction. Let $B$ be a fundamental domain of $\mathcal R_{\restriction A}$, then we may identify the full group of $\mathcal R_{\restriction A}$ with the group of measurable maps from $B$ to the symmetric group on $n$ elements. Using the signature and $\pi_A$, we get a non-trivial morphism $\varphi$ from $[\mathcal R]$ to the group $\LL^0(B,\Z/2\Z)$ of measurable maps from $B$ to $\Z/2\Z$. Such a group has only elements of order 2, so it is a $\Z/2\Z$ vector space. Fix a non-trivial element $y$ in the image of $\varphi$, then using the axiom of choice we find a projection $p: \LL^0(B,\mathfrak S_2)\to\Z/2\Z$ onto the vector space spanned by $y$. Then $p\circ \varphi$ is a non-trivial morphism from $[\mathcal R]$ to $\Z/2\Z$.
\end{proof}

Note that this theorem implies that if $\mathcal R$ is not aperiodic, then its full group equipped with the \textit{uniform} metric does not satisfy the automatic continuity property, as it is connected but admits a morphism onto $\Z/2\Z$. It would be interesting to characterize the full group of $\mathcal R$ equipped with the uniform metric in terms of the automatic continuity property, namely to answer the following question.

\begin{qu} Let $\mathcal R$ be a pmp aperiodic equivalence relation. Is every morphism from $([\mathcal R],d_u)$ to a separable group continuous?
\end{qu}

\section{Extreme amenability}\label{exam}

This section is devoted to giving a simple proof of the following proposition, which was proved by Giordano and Pestov in the ergodic case \cite[Prop. 5.3]{MR2311665}. 

\begin{prop}\label{extam}Let $\mathcal R$ be a hyperfinite equivalence relation, then its full group is extremely amenable.
\end{prop}

The previous proposition is a straightforward application of the two following lemmas. The first one is an immediate consequence of Theorem \ref{ktdense}, but we give here a direct proof.
\begin{lem}
Let $\mathcal R=\bigcup_n \mathcal R_n$ be a pmp equivalence relation, where the $\mathcal R_n$'s are increasing finite equivalence relations. Then $\bigcup_n[\mathcal R_n]$ is dense in $[\mathcal R]$.
\end{lem}
\begin{proof}
Let $\varphi\in [\mathcal R]$ and $\epsilon>0$, since the $\mathcal R_n$'s exhaust $\mathcal R$ there exists $n\in\N$ such that $\mu(\{x\in X: \varphi(x)\, \mathcal R_n\, x\})>1-\epsilon$. Now put $\psi(x)=\varphi(x)$ whenever $\varphi(x)\,\mathcal R_n\,x$. We can extend $\psi$ to an element $\tilde\psi$ of the full group of $\mathcal R_n$. Indeed, $\psi$ induces a partial bijection in every $\mathcal R_n$-class, and every such partial bijection can be extended in a Borel way to have full domain (using a Borel total order on $X$, one can for instance send the first element of each equivalence class not in the domain of $\psi$ to the first element of the class not in the range of $\psi$, etc.). By definition, $d_u(\tilde\psi,\varphi)<\epsilon$.
\end{proof}

\begin{lem}
Let $\mathcal R$ be a finite equivalence relation. Then its full group is extremely amenable.
\end{lem}
\begin{proof}
For $n\in\N$, let $X_n$ be the set of $x\in X$ whose $\mathcal R$-class has cardinality $n$. Then we have a topological group isomorphism between the full group of $\mathcal R$ and the product
$\prod_{n\in\N}[\mathcal R_{\restriction X_n}]$. Let $A_n$ be a transversal for $\mathcal R_{\restriction X_n}$, then we can identify $[\mathcal R_{\restriction X_n}]$ with the group $\LL^0(A_n,\mathfrak S_n)$ of measurable maps from $A_n$ to $\mathfrak S_n$. This identification is an isometry if we put on $\LL^0(A_n,\mathfrak S_n)$ the $\LL^1$ metric $d_1$ defined by:
$$d_1(f,g)=\int_{A_n}d_{\mathrm{Ham}}(f(x),g(x))d\mu(x),$$
where $d_{\mathrm{Ham}}$ is the Hamming metric on $\mathfrak S_n$ (i.e. $d_{\mathrm{Ham}}(\sigma,\sigma')=\mathrm{Card}(\{x\in\{1,...,n\}: \sigma(x)\neq\sigma'(x)\})$). Then, a theorem of Glasner (\cite[Thm. 1.3]{MR1617456}, see also \cite[Thm. 2.20]{MR2311665}) implies that $[\mathcal R_{\restriction X_n}]=\LL^0(A_n,\mathfrak S_n)$ is extremely amenable since $\mathfrak S_n$ is compact. We conclude that $[\mathcal R]=\prod_{n\in\N}[\mathcal R_{\restriction X_n}]$ is also extremely amenable.
\end{proof}

\begin{rmq}In order to prove Proposition \ref{extam}, one could also use Dye's theorem and identify the full group of the aperiodic part $\mathcal R_{\infty}$ of $\mathcal R$ with $\LL^0(Y_{\mathcal R_\infty},[\mathcal R_0])$, which is easily seen to be extremely amenable once we know that $[\mathcal R_0]$ is extremely amenable. Such a proof actually yields that the full group of any hyperfinite equivalence relation is a Levy group. 
\end{rmq}

Using the theorem of Connes, Feldmann and Weiss, and the fact that the full group acts continuously by affine isometric transformations on the field of invariant means on $\mathcal R$, one can easily see that the converse of Proposition \ref{extam} holds, and so we get the following corollary.

\begin{thm}\label{caraextam}Let $\mathcal R$ be a pmp equivalence relation. Then the following statements are equivalent.
\begin{enumerate}[(i)]
\item $\mathcal R$ is hyperfinite.
\item The full group of $\mathcal R$ is amenable.
\item The full group of $\mathcal R$ is extremely amenable.
\end{enumerate}
\end{thm}

\bibliographystyle{alpha}
\bibliography{/Users/francoislemaitre/Dropbox/Maths/biblio}

\end{document}